\documentclass[11pt, reqno]{amsart}
\usepackage{latexsym, amsmath, amssymb, amsthm, verbatim}
\usepackage{cases}
\usepackage[colorlinks]{hyperref}
\usepackage{hypernat}
\usepackage{xcolor} 
\usepackage{color}

\usepackage[margin=1.5in]{geometry}

\hypersetup{
    bookmarks=true,         
    unicode=false,          
    pdftoolbar=true,        
    pdfmenubar=true,        
    pdffitwindow=false,     
    colorlinks=false,       
    linkcolor=red,          
    linkbordercolor=red,
    citecolor=green,        
    citebordercolor=green,
    filecolor=magenta,      
    urlcolor=cyan,           
   urlbordercolor={1 1 1}  
}

\newtheorem{thm}{Theorem}
\newtheorem{lem}[thm]{Lemma}
\newtheorem{prop}[thm]{Proposition}
\newtheorem{cor}[thm]{Corollary}

\theoremstyle{remark}
\newtheorem{rmk}[thm]{Remark}
\newtheorem{example}[thm]{Example}

\theoremstyle{definition}
\newtheorem{defi}[thm]{Definition}

\numberwithin{thm}{section}
\numberwithin{equation}{section}

\makeatletter

\newcommand{\Rmnum}[1]{\expandafter\@slowromancap\romannumeral #1@}
\makeatother

\def\R{{\mathbb R}}

\def\O{{\mathcal O}}

\newcommand{\vep}{\varepsilon}
\newcommand{\ol}{\overline}
\newcommand{\ul}{\underline}

\newcommand{\bpm}{\begin{pmatrix}}
\newcommand{\epm}{\end{pmatrix}}
\newcommand{\la}{\langle}
\newcommand{\ra}{\rangle}

\newcommand{\beq}{\begin{equation}}
\newcommand{\eeq}{\end{equation}}

\DeclareMathOperator*{\limsups}{limsup^\ast}
\DeclareMathOperator*{\liminfs}{liminf_\ast}

\newcommand\re{\textcolor{red}}

plus 6pt minus 12pt
\title[Singular Neumann problems]{\protect{Singular Neumann boundary problems for a class of fully nonlinear parabolic equations in one dimension}}

\author{Takashi Kagaya}
\address{Takashi Kagaya, Institute of Mathematics for Industry, Kyushu University, Japan}
\email{kagaya@imi.kyushu-u.ac.jp}

\author{Qing Liu}
\address{Qing Liu, Department of Applied Mathematics, Faculty of Science, Fukuoka University, Japan}
\email{qingliu@fukuoka-u.ac.jp}

\begin{document}

\begin{abstract}
In this paper, we discuss singular Neumann boundary problem for a class of nonlinear parabolic equations in one space dimension. Our boundary problem describes motion of a planar curve sliding along the boundary with a zero contact angle, which can be viewed as a  limiting model for the capillary phenomenon. We study the uniqueness and existence of solutions by using the viscosity solution theory. We also show the convergence of the solution to a traveling wave as time proceeds to infinity when the initial value is assumed to be convex. 
\end{abstract}


\maketitle

\section{Introduction}

\subsection{Background and motivation}
This paper is concerned with singular boundary problem for fully nonlinear differential equation. Let $b>0$ and $f, g \in C(\mathbb{R})$. We are mainly interested in the following equation in one space dimension:
\begin{numcases}{}
u_t-f(g(u_x) u_{xx})=0 & \text{in $(-b, b)\times (0, \infty)$,} \label{flow eq}\\
u_{x}(b, t)= \infty, \quad u_{x}(-b, t)=-\infty  & \text{for any $t>0$,}\label{bdry}\\
u(\cdot , 0)=u_0 & \text{in  $[-b, b]$,}  \label{initial}
\end{numcases}
where $u_0$ is a given continuous function in $[-b, b]$, and $f$ and $g$ are given continuous functions satisfying several assumptions to be elaborated later. 

One typical example of \eqref{flow eq} in our mind is the following equation:
\beq\label{ex flow eq}
u_t-(1+|u_x|^2)^{1-3\gamma \over 2} |u_{xx}|^{\gamma-1} u_{xx}=0  \quad \text{in $(-b, b)\times (0, \infty)$,} 
\eeq
where $\gamma>1$ is given. It is a special case of \eqref{flow eq} with 
\[
f(s)=|s|^{\gamma-1}s, \quad g(s)=(1+s^2)^{1-3\gamma \over {2\re{\gamma}}}
\]
for $s\in \R$. Such an equation has a geometric interpretation: it describes the motion of a graph-like planar curve whose normal velocity equals to the $\gamma$-th power of its curvature. The singular Neumann boundary condition \eqref{bdry}, on the other hand, depicts tangential sliding behavior of the curve along the boundary $x=\pm b$. It is clear that \eqref{ex flow eq} reduces to the standard curvature flow if $\gamma=1$. 

We are interested in this singular boundary value problem, since it is closely related to the so-called capillary phenomenon occurring at the interfaces of liquid, solid or air phases; see the introductory monograph \cite{FiBook} and references therein. The classical theory formulates the phenomenon as an elliptic or parabolic equation equipped with a Neumann or oblique condition determined by the contact angle $\theta$ on the boundary, which in our current setting can be expressed as
\[
u_x(\pm b,t) = \pm \tan \left(\frac{\pi}{2} - \theta\right). 
\]
However, while there is a vast literature on the case of positive contact angles, less is known for the problems with $\theta=0$, especially in the time-dependent setting. 
One can view the zero-angle case as a limiting situation when the interface between the liquid and solid has a very strong adhesion effect. 
This physical background motivates us to study the singular Neumann problem for parabolic equations. 

It is of value to look into \eqref{flow eq}--\eqref{initial} also from the mathematical point of view. Huisken \cite{Hui2} studied smooth solutions of the Neumann problem with $\theta=\pi/2$ for mean curvature flow. Later, Altschuler and Wu \cite{AW1, AW2} investigated a more general boundary condition for $0<\theta\leq \pi/2$ for quasilinear parabolic equations and proved the convergence of smooth solutions to a traveling wave as time proceeds. We refer to \cite{BeNo11, ChGuKo03, ChGu11} etc.\ for further developments on graph-like curvature flow with contact angle conditions and \cite{KKR, St96-2, St96-1, BaDa3, MiTo, Ed20, Kag1} etc.\ for various approaches to mean curvature flow of general surfaces with right contact angle. 

It is worth mentioning that there is another approach, called level-set formulation, to the same geometric evolution problem using the viscosity solution theory; we refer to \cite{GSa, Sa, Ba2, ISa, Gbook} for well-posedness results. The contact angle boundary condition under this formulation reads 
\[
\la\nabla u, \nu\ra=k |\nabla u|,
\]
where $k=\cos\theta$ and $\nu$ denotes the outward unit normal to the boundary. In the references mentioned above, it is also assumed that $0\leq k<1$, i.e., $0<\theta\leq \pi/2$. We are thus curious whether there still exists a unique solution and how it evolves if we take $k=1$. 

The primary novelty of this work is to discuss the contact angle boundary problem in the case $\theta=0$, which is not addressed in the aforementioned papers.  It turns out that this sliding boundary condition requires special structure on the parabolic operator. In fact, for the usual curve shortening flow, i.e., \eqref{ex flow eq} with $\gamma=1$, one cannot expect existence of solutions bounded in space. The well-known ``grim reaper solution'' suggests that in general solutions may blow up at the boundary if we turn the contact angle into zero. In fact, we can really prove the nonexistence of spatially bounded solutions in this case even if the initial condition is bounded; see Theorem \ref{thm:blowup} below. The main reason can be attributed as insufficiency of curvature effects to sustain the adhesion of curve to the boundary. 

It is thus natural to consider geometric flows with stronger curvature effects, which leads us to \eqref{ex flow eq} with $\gamma>1$. The power curvature flow is of independent interest and has important applications in image processing; consult \cite{AGLM, Ca} for instance. Under the singular boundary condition, we focus on the one dimensional case and investigate the well-posedness and large time behavior for nonlinear parabolic equations including \eqref{ex flow eq}. In particular, we attempt to extend,  in our singular and fully nonlinear setting, the known results \cite{Hui2, AW1, AW2} on large time convergence to traveling wave solutions.

We remark that, in the setting of power curvature flow \eqref{ex flow eq} with zero contact angle, it is actually possible to use the level set method to find a solution $U$ by treating the boundary portions above the contact points as a part of the curve motion; namely, one can consider the evolution of $U$ whose zero level set $\{(x, y)\in \R^2: U(x, y, t)=0\}$ for any $t\geq 0$ is equal to 
\[
\left\{(x, u(x, t)): |x|< b\right\}\cup \{(b, y): y\geq u(b, t)\}\cup \{(-b, y): y\geq u(-b, t)\}.
\]
See related results in \cite{Go, ISo}. 
In this work, in order to establish an approach applicable to more general nonlinear diffusions like \eqref{flow eq}, we choose to study the graph case instead of the level set equation.

\subsection{Main results}
We impose the following assumptions on the functions $f$ and $g$.
\begin{itemize}
\item[(A1)] $f \in C(\mathbb{R})$ is strictly increasing with $f(0) = 0$ and $f(s)\to \pm \infty$ as $s \to \pm \infty$. 
\item[(A2)] $g \in C(\mathbb{R})$ is positive and 
\beq\label{exponent alpha}
 |s|^{\alpha}g(s) \to C_\pm \quad \text{as } s \to \pm \infty 
 \eeq
for some exponent $\alpha > 2$ and constants $C_\pm > 0$. 
\end{itemize}

We establish well-posedness results for \eqref{flow eq}--\eqref{initial} under the assumptions (A1) and (A2) as well as the continuity of initial value $u_0$ in $[-b, b]$. 
As is easily seen, there are two major difficulties caused by the structure of the equation: high nonlinearity of the operator due to the appearance of $f$ and strong degeneracy of the elliptic operator because of the rapid decay of $g(s)$ as $s\to \infty$. One therefore cannot expect solutions to be smooth in these circumstances.

 The framework of viscosity solutions (cf. \cite{CIL})  then becomes a natural choice to study \eqref{flow eq}. Singular Neumann problems for nonlinear elliptic equations are studied in  \cite{LaL2} with the boundary condition interpreted in the viscosity sense: there are no $C^2$ functions touching $u$ from below at the boundary.  For our parabolic problem, we apply the same viscosity interpretation to \eqref{bdry}. A precise definition of viscosity solutions is given in Section \ref{sec:def-vis}. 
 
In the context of Cauchy-Dirichlet problems, the derivative blow-up on the boundary is related to the loss of boundary condition. We refer to \cite{FiLi, PoSo1, PoSo2, PoSo3} for detailed analysis on the boundary behavior of such kind of solutions for viscous Hamilton-Jacobi equations. 
Also, \cite{CeNo, MiZh} study generalized Dirichlet problem for mean curvature flow with a driving force and show large time convergence to traveling waves with blow-up at boundary. The problem setting in this work is different from these papers: our equation is of nonlinear curvature type and, throughout the evolution, we impose singular Neumann condition directly without prescribing Dirichlet data at all. 
 
Our first main result is as follows. 
\begin{thm}[Well-posedness for $\alpha>2$]\label{thm existence perron}
Let $b>0$. 
Assume that functions $f, g$ satisfy (A1), (A2) respectively. 
Let $u_0 \in C([-b,b])$.
Then there exists a unique viscosity solution $u\in C([-b, b]\times [0, \infty))$ of \eqref{flow eq}--\eqref{initial}. 
\end{thm}

We use comparison arguments to prove the uniqueness and adopt Perron's method to obtain the existence. We remark that the well-posedness of generalized Dirichlet boundary value problems for nonlinear parabolic equations is established in \cite{BaDa1, BaDa2, AtBa}. 
Our approach share similarities with theirs, especially in the proof of comparison principle. However, the methods proposed in \cite{BaDa1, BaDa2, AtBa} basically apply to quasilinear equations and require the operator to have a linear growth in the second derivative $u_{xx}$. 
In the work, we are able to handle parabolic equations with more general nonlinear structures. 


Heuristically speaking, by (A1) one can get the inverse function $f^{-1}$ of $f$ and then transform \eqref{flow eq} to 
\beq\label{transformed eq}
f^{-1}(u_t)=g(u_{x}) u_{xx}.
\eeq
This transformation allows us to apply standard comparison arguments to estimate the difference between the right hand sides of viscosity inequalities. 
However, in this process one needs to provide a Lipschitz bound on the time derivative $u_t$ to avoid the degeneracy appearing in the difference of the left hand sides; note that $f^{-1}(s_1)-f^{-1}(s_2)\to 0$ fails to hold in general when $|s_1 - s_2| \to 0$ without assuming the boundedness of $s_1, s_2$.

A first and crucial step in our proof of the comparison principle is to adopt the sup-convolution for the subsolution in time. This enables us to obtain its Lipschitz continuity in time, which facilitates treating the left hand side of \eqref{transformed eq}. 
Also, our argument to handle the singular boundary condition resembles the technique in \cite{So} used to study the state constraint boundary problem. 
See Section \ref{sec:comparison} for a detailed proof of the comparison principle.


We obtain the existence result by adapting Perron's method to our singular boundary problem. Our proof consists of two steps. We first study the existence problem under the following assumption on $u_0$:
\begin{itemize}
\item[(A3)] $u_0 \in C^{1,1}((-b,b)) \cap C([-b,b])$ and 
\begin{align}
& L_- \leq f(g((u_0)_x)(u_0)_{xx}) \leq L_+ \quad \text{a.e. in} \; \;  (-b, b), \label{initial cond1} \\
& (u_0)_x(x)\to \pm \infty \quad \text{as} \ x\to \pm b \label{initial cond2}
\end{align}
for some $L_\pm \in \mathbb{R}$.
\end{itemize}
 This additional regularity assumption allows us to construct with ease a sub- and a supersolution satisfying the initial condition. Then the existence of solutions in between follows from Perron's method. In the general case when $u_0$ is only assumed to be continuous in $[-b, b]$, we can approximate the initial value by a sequence of smooth functions that satisfy (A3) and then apply the standard stability argument to show the local uniform convergence of the approximate solutions. 
The regularization of $u_0$ will be introduced in detail in Lemma \ref{lem regularization}.  

We remark that the $C^{1, 1}$ regularity assumed in (A3) implies local semiconvexity and semiconcavity of $u_0$ in $(-b, b)$ while \eqref{initial cond1} additionally imposes a precise requirement on the semiconvexity and semiconcavity constants. In fact, by (A1) and (A2), \eqref{initial cond1} can be rewritten as 
\[
{f^{-1}(L_-)\over g((u_0)_x)}\leq (u_0)_{xx}\leq {f^{-1}(L_+)\over g((u_0)_x)},
\]
which specifies pointwise bounds of semiconvexity and semiconcavity constants of $u_0$ in terms of $g((u_0)_x)$. In particular, if (A3) holds with $L_-\geq 0$, then $u_0$ is convex in $(-b, b)$. In addition, due to \eqref{exponent alpha} and \eqref{initial cond2}, (A3) with $L_->0$ implies that $(u_0)_{xx}(x)\to \infty$ as $x\to \pm b$.


It is worth pointing out that the range $\alpha>2$ in (A2) plays a key role for the existence of spatially bounded solutions. This condition is consistent with our observation that only with an exponent $\gamma>1$ does \eqref{ex flow eq} admit solutions that are continuous in $[-b, b]\times [0, \infty)$ and satisfy \eqref{bdry}. Indeed, if we assume in place of (A2) that
\begin{itemize} 
\item[(A2')] $g \in C(\mathbb{R})$ is positive and 
\[ |s|^{\alpha}g(s) \to C_\pm \quad \text{as } s \to \pm \infty \]
for some exponent $\alpha \le 2$ and positive constants $C_\pm > 0$,
\end{itemize}
then we have the following nonexistence result. 

\begin{thm}[Nonexistence for $\alpha\leq 2$]\label{thm:blowup}
Let $b>0$. Assume that functions $f$ and $g$ satisfy (A1) and (A2'). Let $u_0\in C([-b, b])$. Then there exist no solutions of \eqref{flow eq}--\eqref{initial} in $C([-b, b]\times [0, \infty))$. 
\end{thm}

Our proof is based on explicit construction of a sequence of subsolutions $\ul{u}(\cdot; k)$, given in \eqref{def:lower-blowup}, which are uniformly bounded in $[-b, b]$ at $t=0$ but diverges to infinity on the boundary instantaneously after the initial moment as $k\to \infty$. 
Roughly speaking, in order to build such $\ul{u}(\cdot; k)$ growing rapidly on the boundary, we essentially want to find a function $\phi\in C^2((-b, b))$ such that as $x\to \pm b$, the following two properties hold:
\begin{itemize}
\item[(a)] $\phi(x)\to \infty$;
\item[(b)] $f\left(g(\phi_x(x))\phi_{xx}(x)\right)\to \infty$.
\end{itemize}
The critical case when $\alpha = 2$ is particularly difficult to handle. Note that in this case, while the property (b) holds,  (a) fails if we choose $\phi(x) = (|b|-|x|)^\beta$ near $x=\pm b$ with $\beta \in (0,1)$; on the other hand, if we let $\phi(x) = -\log(|b|-|x|)$ near $x=\pm b$, then (a) holds but (b) fails. 
It turns out that $\phi(x) = -\log(\log(|b|-|x|))$ satisfies both properties and thus becomes an appropriate choice for our purpose. See more rigorous details for the construction of $\ul{u}(\cdot, k)$ based on this $\phi$ in Section \ref{sec:nonexistence}. 

Another goal of this work is to understand the large time asymptotics for \eqref{flow eq}--\eqref{initial}. We aim to show the convergence of the solution to the one of traveling wave type satisfying \eqref{bdry}, which generalizes the results in \cite{AW1, AW2} for fully nonlinear diffusions including power curvature flows. To this end, we begin with solving 
the stationary problem
\begin{numcases}{}
c = f(g(W_x)W_{xx}) \qquad \text{in $(-b,b)$}, \label{traveling eq} \\
W_x(b) = \infty, \quad W_x(-b) = -\infty, \label{traveling bdry}
\end{numcases}
where $W\in C([-b,b])$ and $c\in \R$ are both unknowns. 
This problem fully characterizes the traveling wave solutions of \eqref{flow eq}--\eqref{initial}, as $c\in \R$ stands for the traveling speed while $W$ gives the traveling profile. 

It turns out that there exists a unique $c>0$ that yields a unique solution $W$ up to constants. 
\begin{thm}[Solvability of stationary problem]\label{thm:ex-tw}
Let $b>0$. 
Assume that functions $f$ and $g$ satisfy (A1) and (A2). 
Then there exists a pair 
\[
(W, c) \in \{C([-b,b]) \cap C^2((-b,b))\} \times (0, \infty)
\] 
satisfying  \eqref{traveling eq} and \eqref{traveling bdry}.
Moreover, the following properties hold. 
\begin{enumerate}
\item[(i)] If $(\tilde{W}, \tilde{c})\in \{C([-b,b]) \cap C^{0,1}((-b,b))\} \times \mathbb{R}$ is another pair satisfying \eqref{traveling eq} and \eqref{traveling bdry}  in the viscosity sense,  then $\tilde{c}=c$ holds and there exists a constant $a\in \R$ such that 
\[
\tilde{W} = W+a \quad \text{in $[-b, b]$.}
\]
\item[(ii)] The profile function $W$ is strictly convex and of class $C^{\frac{\alpha-2}{\alpha-1}}([-b,b])$. 
\end{enumerate} 
\end{thm}

The solvability of \eqref{traveling eq} and \eqref{traveling bdry} will be discussed by using the shooting method.  Roughly speaking, we can use the inverse function $f^{-1}$ to rewrite the equation as 
\beq\label{transformed eq2}
f^{-1}(c)=g(W_x)W_{xx}=G(W_x)_x,
\eeq
where $G$ is an antiderivative of $g$. In order to handle this ordinary differential equation, we set the function class to be $C([-b,b]) \cap C^{0,1}((-b,b))$. 
Noticing that
\[
W_{xx}={f^{-1}(c)\over g(W_x)}
\]
holds at least formally, we can get a local bound of $W_{xx}$ for $W\in C^{0,1}((-b,b))$, which enables us to integrate \eqref{transformed eq2} in $x$. 
Then it becomes easier for us to use the integral form to find a solution $(W, c)$ and to prove the uniqueness in the sense of (i).  

We finally show the large time behavior of the unique solution of \eqref{flow eq}--\eqref{initial}. 
\begin{thm}[Large time asymptotics]\label{thm:conver-sol}
Let $b>0$. 
Assume that functions $f$ and $g$ satisfy (A1) and (A2). 
Assume also that $f^{-1}$ is Lipschitz away from $s=0$, $g$ is Lipschitz in $\mathbb{R}$ and $u_0 \in C([-b,b])$ is convex. 
Let $u$ and $(W,c)$ be respectively the solution of \eqref{flow eq}--\eqref{initial} obtained in Theorem \ref{thm existence perron} and the solution of \eqref{traveling eq} and \eqref{traveling bdry} obtained in Theorem \ref{thm:ex-tw}. 
Then there exists a constant $m \in \mathbb{R}$ such that 
\beq\label{convergence intro}
 \sup_{x\in [-b, b]} \left|u(x, t) - (W(x) + m + ct)\right| \to 0 \quad \text{as} \quad t \to \infty. 
\eeq
\end{thm}

It turns out that, in order to prove Theorem \ref{thm:conver-sol}, one can focus on the case when $u_0$ satisfies (A3) with $L_->0$, since the result for general continuous convex initial values can be obtained by utilizing the approximation argument analogous to that in the proof for the existence theorem. 

Under (A3) and the condition $L_->0$, we can obtain
several uniform regularity estimates, which are important for us to carry out compactness arguments. 
The assumption (A3) enables us to deduce the Lipschitz continuity of $u$ global in time. Indeed, we can use (A3) to show that
\beq\label{formal time lip}
L_-\leq u_t\leq L_+
\eeq
in the viscosity sense. 

Moreover, the additional condition $L_->0$, related to the convexity of $u_0$ in the general case, 
largely helps us obtain H\"{o}lder continuity of $u(\cdot, t)$ in $[-b, b]$ uniformly for all $t\geq 0$. The uniform H\"{o}lder regularity further enables us to apply the Ascoli-Arzel\`{a} theorem to find a convergent subsequence of $u(\cdot, t)-ct$ as $t\to \infty$. The full convergence is shown by adopting a strong comparison principle, which again relies on \eqref{transformed eq} to avoid the high nonlinearity and degeneracy of the original equation. Using \eqref{transformed eq}, we can obtain the strong comparison by adapting the method in \cite{Dal}. 

To further see the connection of the assumption $L_->0$ with the spatial regularity, let us use 
\eqref{transformed eq} to rewrite our equation again. In fact, by \eqref{formal time lip} it is easily seen that any solution of \eqref{flow eq} is a subsolution of 
\[
-u_{xx}+ {f^{-1}(L_-)\over g(u_x)}\leq 0,
\]
which is of viscous Hamilton-Jacobi type. In \cite{CaLePo, Ba4, CadSi, ArTr}, the H\"{o}lder regularity of viscosity subsolutions are established for such kind of nonlinear equations, but all of these results require a strong coercivity assumption on the gradient term, 
which is equivalent to the condition $f^{-1}(L_-)>0$ in our circumstances. This somehow explains why we also need the assumption $L_->0$ or, more generally, the convexity of $u_0$ for our asymptotic analysis. It would be interesting to obtain the same large time behavior without assuming the convexity of $u_0$. We do expect that the solution eventually becomes convex in space even if it is initially not convex.

\subsection{Organization of the paper} The rest of the paper is organized in the following way. We first introduce the definition of viscosity solutions of the singular Neumann boundary problems in Section \ref{sec:def-vis}. Section \ref{sec:comparison} is devoted to the proof of the comparison principle.  In Section \ref{sec:existence}, we show the existence of solutions under (A2) and prove the nonexistence result under (A2'). We give a detailed proof of Theorem \ref{thm:ex-tw} for the stationary problem in Section \ref{sec:stationary} and show the large time behavior as in Theorem \ref{thm:conver-sol} in Section \ref{sec:asymptotics}. 

\subsection*{Acknowledgments} 
The authors would like to thank the anonymous referees for providing very helpful comments. 
The work of the first author is supported by JSPS Grants-in-Aid for Scientific Research, No.\ JP19K14572, JP18H03670 and JP20H01801.
The work of the second author is supported by JSPS Grant-in-Aid for Scientific Research, No.\ JP19K03574.

\section{Definition of viscosity solutions}\label{sec:def-vis}

We need to adopt the viscosity solution theory \cite{CIL, Gbook} to handle \eqref{flow eq}, which is in general fully nonlinear and degenerate parabolic. 
We also refer to \cite{LaL2} for introduction to the singular boundary condition in the viscosity sense. For convenience of notation, hereafter we denote
\[
Q:=(-b, b)\times (0, \infty), \quad \ol{Q}_T:=[-b, b]\times [0, T)
\]
for any $T>0$ and let 
\[
F(p, z):=-f(g(p)z)
\]
for $p, z\in \R$.


\begin{defi}[Solutions of singular Neumann problem]\label{def singular}
An upper semicontinuous function $u$ in $\ol{Q}$ is called a subsolution of \eqref{flow eq} and \eqref{bdry} if whenever there exist $(x_0, t_0)\in Q$ and $\phi\in C^2(\ol{Q})$ such that $u-\phi$ attains a local maximum at $(x_0, t_0)$, we have 
\[
\phi_t(x_0, t_0)+F\left(\phi_x(x_0, t_0), \phi_{xx}(x_0, t_0)\right)\leq 0.
\]
A lower semicontinuous function $u$ in $\ol{Q}$ is called a supersolution of \eqref{flow eq} and \eqref{bdry} if the following conditions hold:
\begin{enumerate}
\item[(i)] Whenever there exist $(x_0, t_0)\in Q$ and $\phi\in C^2(\ol{Q})$ such that $u-\phi$ attains a local minimum at $(x_0, t_0)$, we have 
\[
\phi_t(x_0, t_0)+F\left(\phi_x(x_0, t_0), \phi_{xx}(x_0, t_0)\right)\geq 0.
\]
\item[(ii)] For any $t>0$ and any function $\phi\in C^2(\overline{Q})$, $u-\phi$ never attains a local minimum at the point $(b, t)$ or $(-b, t)$.  
\end{enumerate}
A continuous function $u$ in $\ol{Q}$ is said to be a solution of \eqref{flow eq} and \eqref{bdry} if $u$ is both a subsolution and a supersolution of \eqref{flow eq} and \eqref{bdry}.
\end{defi}

The condition (ii) amounts to saying that supersolutions cannot be tested from below by any smooth functions at boundary points. On the other hand, we do not impose any condition on the boundary behavior of subsolutions. 

One can similarly define solutions of the corresponding stationary problem \eqref{traveling eq}--\eqref{traveling bdry} for a given $c\in \R$.

\begin{defi}[Solutions of stationary singular Neumann problem]\label{def singular stationary}
An upper semicontinuous function $W$ in $[-b, b]$ is called a subsolution of \eqref{traveling eq}--\eqref{traveling bdry} if whenever there exist $x_0\in (-b, b)$ and $\phi\in C^2([-b, b])$ such that $W-\phi$ attains a local maximum at $x_0\in (-b, b)$, we have 
\[
c+F\left(\phi_x(x_0), \phi_{xx}(x_0, t_0)\right)\leq 0.
\]
A lower semicontinuous function $W$ in $\ol{Q}$ is called a supersolution of \eqref{traveling eq} and \eqref{traveling bdry} if the following conditions hold:
\begin{enumerate}
\item[(i)] Whenever there exist $x_0\in Q$ and $\phi\in C^2([-b, b])$ such that $W-\phi$ attains a local minimum at $x_0$, we have 
\[
c+F\left(\phi_x(x_0), \phi_{xx}(x_0)\right)\geq 0.
\]
\item[(ii)] For any function $\phi\in C^2([-b, b])$, $W(x)-\phi(x)$ never attains a local minimum at $x=b$ or $x=-b$. 
\end{enumerate}
A function $W\in C([-b, b])$ is called a solution of \eqref{traveling eq} and \eqref{traveling bdry} if $W$ is both a subsolution and a supersolution of \eqref{traveling eq} and \eqref{traveling bdry}.
\end{defi}

As a standard remark in the viscosity solution theory, in Definition \ref{def singular}, 
we may only consider test functions $\phi$ such that $u-\phi$ attains a strict maximum or minimum at $(x_0, t_0)$ by adding the quartic function $a(x-x_0)^4$ with $a\in \R$. Moreover, we may use semijets $\overline{P}^\pm u(x_0, t_0)$ to rewrite the definitions; see \cite{CIL} for more details on the semijets. 

\section{Comparison principle}\label{sec:comparison}

In this section we provide a comparison result. We will use the comparison theorem in Section \ref{sec:existence}.
It will also be applied to study large time behavior under the assumptions (A1)--(A3). 


\begin{thm}[Comparison principle]\label{thm cp}
Let $b>0$. 
Assume that $f\in C(\R)$ satisfies (A1) and $g\in C(\R)$ is nonnegative. 
Let $u$ and $v$ be respectively a sub- and a supersolution of \eqref{flow eq} and \eqref{bdry}. 
Assume that $u(\cdot, 0)$ is continuous in $[-b, b]$ and $u$ is locally bounded in $\ol{Q}=[-b, b]\times [0, \infty)$ in the sense that
\begin{equation}\label{locally_bdd} 
\|u\|_{L^\infty((-b,b) \times (0,T))} < \infty \quad \text{ for any $T\in (0, \infty)$}. 
\end{equation}
Assume in addition that $u$ is accessible on the boundary; namely, for any $t\in (0, \infty)$,
\beq\label{space conti}
\limsup_{(x, s)\to (b, t)}  u\left(x, s\right) =  u(b, t)\quad \text{and} \quad \limsup_{(x, s)\to (-b, t)}  u\left(x, s\right) =  u(-b, t).
\eeq
If 
\beq\label{initial comparison}
u(\cdot, 0)\leq v(\cdot, 0) \quad \text{ in $[-b, b]$},
\eeq
 then $u\leq v$ in $\ol{Q}$.
\end{thm}


\begin{rmk}\label{rmk:com-bdry}
The accessibility assumption \eqref{space conti} is necessary for Theorem \ref{thm cp}. 
Suppose that $u$ and $v$ are respectively a subsolution and a supersolution of \eqref{flow eq}--\eqref{bdry} fulfilling \eqref{initial comparison}. 
Then we can choose some $t_0>0$ and $\mu_\pm\in C([0,\infty))$ satisfying $\mu_{\pm}(0)=u_0(\pm b)$, 
\[
\mu_{+}(t)\geq u(b, t), \quad \mu_{-}(t)\geq u(-b, t)
\]
for all $0< t\leq t_0$, and
\[
\mu_+(t) >\max\{u(b, t),\ v(b, t)\}, \quad \mu_-(t) >\max\{u(-b, t),\ v(-b, t)\}
\]
for all $t> t_0$.  We construct another subsolution
\[ \tilde{u}(x,t) := 
\begin{cases}
u(x,t) & \mbox{if} \quad (x, t)\in (-b, b)\times [0, \infty),\\ 
\mu_+(t) & \mbox{if} \quad x=b, \; \; t \ge 0, \\
\mu_-(t) & \mbox{if} \quad x=-b, \; \; t\ge 0. 
\end{cases} \]

Such a subsolution clearly fails to satisfy \eqref{space conti}. 
Therefore the statement of the comparison principle cannot hold without assuming \eqref{space conti}.
\end{rmk}

\begin{proof}[Proof of Theorem \ref{thm cp}]
Our proof consists of several steps. 

\ul{\textit{Step 1. Sup-convolution}}

We first take the sup-convolution for $u$ with respect to the time variable. Fix $T > 0$ arbitrarily.  For any $\delta>0$ and $(x, t)\in \ol{Q}_T$, let 
\[
u_\delta(x, t)=\sup_{s\in [0, T+1]} \left\{u(x, s)-{|t-s|^2\over \delta}\right\}.
\]
By the definition, it is clear that 
\begin{equation}\label{remove_lipa}
u\leq u_{\delta_1}\leq u_{\delta_2} \quad \text{in} \; \;  \ol{Q}_T \quad \text{for all} \; \; 0<\delta_1\leq \delta_2. 
\end{equation} 
Due to the local boundedness \eqref{locally_bdd} and upper semicontinuity of $u$, there exists $s_\delta(x, t)\in [-b, b]\times [0, T+1]$ such that 
\beq\label{sup-convolution max}
u_\delta(x, t)=u(x, s_\delta(x, t))-{|t-s_\delta(x, t)|^2\over \delta}.
\eeq
Since $u_\delta(x, t)\geq u(x, t)$, we have 
\begin{equation}\label{remove_lipb}
|t - s_\delta(x, t)| \le  C_T \sqrt{\delta}, 
\end{equation}
where 
\beq\label{comparison constant1}
C_T=(2\|u\|_{L^\infty(\ol{Q}_{T+1})})^{1\over 2}.
\eeq
Hereafter we take $\delta>0$ satisfying 
\beq\label{delta range}
C_T\sqrt{\delta}<\min\{1,\ T\}.
\eeq
It is also not difficult to see that $u_\delta$ is upper semicontinuous and \eqref{space conti} holds for $u_\delta$, that is, 
\beq\label{space conti delta}
\limsup_{(x, s)\to (\pm b, t)}  u_\delta\left(x, s\right) =  u_\delta(\pm b, t)\quad \text{ for any $t\in (0, T)$.}
\eeq
For instance, to see \eqref{space conti delta} at $(b, t)$ for some $t<T$, we take $s_\delta(b, t)$ satisfying \eqref{sup-convolution max} with $x=b$ and use \eqref{space conti} to find a sequence $(y_k, \tau_k)\in (-b, b)$ such that $(y_k, \tau_k)\to (b, s_\delta(b, t))$ and 
$u(y_k, \tau_k)\to u(b, s_\delta(b, t))$
as $k\to \infty$. By the definition of $u_\delta$, we get
\[
u_\delta(y_k, t+\tau_k-s_\delta(b, t))\geq u(y_k, \tau_k)-{|t-s_\delta(b, t)|^2\over \delta}.
\]
Letting $k\to \infty$, we obtain
\[
\limsup_{k\to \infty} u_\delta(y_k, t+\tau_k-s_\delta(b, t))\geq u_\delta(b, t),
\]
which implies the desired accessibility of $u_\delta$ at $(b, t)$.

From the continuity of $u(\cdot,0)$, we get 
\beq\label{remove lip4}
\sup_{x\in [-b, b]} \left(u_\delta(x, 0)-u(x, 0)\right)\leq \omega(\delta)
\eeq
when $\delta>0$ is small, where $\omega$ is a modulus of continuity. 
Indeed, let $x_\delta \in [-b,b]$ be a maximizer of $u_\delta(\cdot, 0)-u(\cdot, 0)$ in $[-b, b]$.  Then using \eqref{remove_lipb}, the upper continuity of $u$ and the continuity of $u(\cdot, 0)$, we have 
\[
\limsup_{\delta\to 0}\sup_{x\in [-b, b]}(u_\delta(x, 0)-u(x, 0))\leq \limsup_{\delta\to 0}\left(u(x_\delta, s_\delta(x_\delta,0))-u(x_\delta, 0)\right)\leq 0,
\]
which yields \eqref{remove lip4} immediately.

Moreover, we can apply a standard argument to show that $u_\delta$ is a subsolution in $Q_{T, \delta}:=~(-b, b)\times (C_T\sqrt{\delta}, T)$. 
Indeed, if there exists $\phi\in C^2(\ol{Q})$ such that $u_\delta-\phi$ attains a local maximum at some $(x_0, t_0)\in Q_{T, \delta}$ with $\phi(x_0, t_0)=u_\delta(x_0, t_0)$, then 
\beq\label{remove lip2}
(x, t, s) \mapsto u(x, s)-{|t-s|^2\over \delta}-\phi(x, t)
\eeq
attains a local maximum in $(-b,b) \times (C_T\sqrt{\delta}, T) \times (0,T+1)$ at $(x_0, t_0, s_\delta(x_0, t_0))$. 
Since $t_0 > C_T\sqrt{\delta}$, it follows from \eqref{remove_lipb} and \eqref{delta range} that $0<s_\delta(x_0, t_0) < T+1$. 

Noticing that
\[
(x, s)\mapsto u(x, s)- {|t_0-s|^2\over \delta}-\phi(x, t_0)
\]
attains a maximum at $(x_0, s_\delta(x_0, t_0))$, we may apply the definition of subsolutions to get 
\beq\label{remove lip3}
{2(s_\delta(x_0, t_0)-t_0)\over \delta}+F(\phi_x(x_0, t_0), \phi_{xx}(x_0, t_0))\leq 0.
\eeq
On the other hand, the maximality in \eqref{remove lip2} also implies that 
\[
{2(s_\delta(x_0, t_0)-t_0)\over \delta}=\phi_t(x_0, t_0),
\]
which, combined with \eqref{remove lip3}, completes the verification of the subsolution property of $u_\delta$ at every $(x_0, t_0)\in Q_{T, \delta}$.

Another important property of $u_\delta$ is its Lipschitz continuity in time; for any fixed $x\in [-b, b]$, 
\beq\label{time lip}
|u_\delta(x, t_1)-u_\delta(x, t_2)|\leq L_\delta |t_1-t_2| \quad \text{for} \; \; t_1, t_2 \in [0,T], \; \; x \in [-b,b],  
\eeq
where $L_\delta>0$ is a constant depending only on $T > 0$ and $\delta>0$. 
Indeed, by direct calculation we obtain 
\[\begin{aligned}
&\; u_\delta(x, t_1) - u_\delta(x,t_2) \\
\le&\; \left\{u(x, s_\delta(x,t_1)) - {|t_1-s_\delta(x,t_1)|^2\over \delta}\right\} - \left\{u(x, s_\delta(x,t_1)) - {|t_2-s_\delta(x,t_1)|^2\over \delta}\right\} \\
=&\; \frac{1}{\delta}(t_1+t_2 - s_\delta(x,t_1)) (t_2 - t_1) \le \frac{3T +1}{\delta} |t_1 - t_2|. 
\end{aligned}\]
Interchanging the roles of $t_1$ and $t_2$ enables us to obtain a symmetric estimate. We thus get \eqref{time lip} with
$L_\delta=(3T+1)/\delta.$

\ul{\textit{Step 2. Choice of $\delta$ for comparison near t=0}}

Let us now proceed to the main part of the proof. 
Assume by contradiction that  $u-v$ takes a positive value somewhere in $Q$. 
Then, there exists $T>0$ large, $\mu>0$ such that 
\[ \max_{\ol{Q}_T} \left(u(x,t) - v(x,t) - \frac{1}{T-t} \right) > \mu. \]
Fix such $T>0$ and let $u_\delta$ be the sup-convolution of $u$ with this $T$, where $\delta>0$ is taken to satisfy \eqref{delta range}.

It follows from the monotonicity \eqref{remove_lipa} that there exists $(\hat{x}, \hat{t})\in \ol{Q}_T$ such that
\beq\label{contra cond}
\max_{\ol{Q}_T} \Psi_\delta =\Psi_\delta(\hat{x}, \hat{t})> \mu
\eeq
for any $\delta > 0$, where 
\[
\Psi_\delta(x, t)=u_\delta(x, t)-v(x, t)-{1 \over T-t}. 
\]

By the upper semicontinuity of $\Psi_\delta$, we have  
\[
\limsup_{t\to 0} \sup_{x\in [-b, b]} \Psi_\delta (x, t)\leq \sup_{x\in [-b, b]}\Psi_\delta(x, 0)= \sup_{x\in [-b, b]} (u_\delta(x, 0)-v(x, 0))-{1 \over T}.
\]
In view of \eqref{remove lip4} and \eqref{initial comparison}, we can take $\delta_1>0$ sufficiently small to deduce that
\[
\sup_{x\in [-b, b]}\Psi_{\delta_1}(x, 0) \leq \sup_{x\in [-b, b]} (u(x, 0)-v(x, 0))+\omega(\delta_1)-{1 \over T}<-{1 \over 2T}.
\]
Then there exists $t_1 > 0$ such that 
\[ \sup_{ [-b,b] \times [0, t_1]} \Psi_{\delta_1} \le - \frac{1}{4T}. \]
Taking \eqref{delta range} into consideration, we choose $\delta>0$ such that 
\beq\label{delta range2}
\delta<\delta_1, \quad C_T\sqrt{\delta}<\min\left\{1,\ T,\ {t_1\over 2}\right\}, 
\eeq
where $C_T>0$ is a constant given in \eqref{comparison constant1}. We thus obtain by \eqref{remove_lipa}
\begin{equation}\label{remove_lipd}
\sup_{[-b,b] \times [0, \ 2C_T\sqrt{\delta}]} \Psi_{\delta}\le - \frac{1}{4T}.
\end{equation}
Let us fix $\delta >0$ as in \eqref{delta range2} and discuss two cases. 

\ul{\textit{Step 3. Scaling argument and comparison near boundary}}

Case 1. Suppose that there exists a maximizer $(\hat{x}, \hat{t})$ of $\Psi_\delta$ satisfying $\hat{x}=b$ or $\hat{x}=-b$. 
Without loss of generality, we may assume that $\hat{x}=b$; the case $\hat{x}=-b$ is essentially the same. 

By applying \eqref{contra cond} with $\hat{x} = b$ and \eqref{space conti delta}, we then obtain 
\begin{equation}\label{range lambda0}
\max_{\ol{Q}_T} \Psi_{\delta, \lambda}\geq {\mu\over 2}>0
\end{equation}
for $\lambda\in (0, 1)$ close to $1$, where
\[
\Psi_{\delta, \lambda}(x, t)=u_\delta(\lambda x, t)- \lambda v(x, t)-{1 \over T-t}
\]
 We take $\lambda\in (0, 1)$ sufficiently close to $1$ to also satisfy
\beq\label{range lambda1}
\left({1\over \lambda} - 1\right)L_\delta \leq {1 \over 2T^2}
\eeq
and 
\beq\label{range lambda2}
\min_{|s| \le L_\delta} \left\{\lambda f^{-1}(s) - f^{-1}\left(s - \dfrac{1}{2T^2}\right)\right\}>0. 
\eeq
Recall that $L_\delta>0$ appeared in \eqref{time lip}. 
Also, by the semicontinuity of $u_\delta$ and $v$, it follows from \eqref{remove_lipd} that 
\[ \limsup_{\lambda \to 1} \sup_{[-b,b] \times [0, 2C\sqrt{\delta}]} \Psi_{\delta, \lambda} \le \sup_{[-b,b] \times [0, 2C\sqrt{\delta}]} \Psi_{\delta} \le - \frac{1}{4T}. \]
This amounts to saying that
\begin{equation}\label{range lambda3}
\sup_{[-b,b] \times [0, 2C\sqrt{\delta}]} \Psi_{\delta, \lambda} \le 0
\end{equation}
when $\lambda$ is chosen to be close to $1$. 
We fix $\lambda\in (0, 1)$ that satisfies \eqref{range lambda0}--\eqref{range lambda3}. 

Let us now double the space variables. Consider, for any $\vep>0$ small, 
\[
\Phi_{\lambda, \vep}(x, y, t):=u_\delta(x, t)-\lambda v(y, t)-{|x- \lambda y|^2\over \vep}-{1 \over T-t}
\]
for $(x, y)\in [-b, b]$ and $t\in [0, T)$.
Since 
\beq\label{remove lip5}
\Phi_{\lambda, \vep}(\lambda \ol{x}, \ol{x}, \ol{t})=\Psi_{\delta, \lambda}(\ol{x}, \ol{t})=\max_{\ol{Q}_T}\Psi_{\delta, \lambda}>0,
\eeq
$\Phi_{\lambda, \vep}$ attains a positive maximum at some $(x_\vep, y_\vep, t_\vep)\in [-b, b]^2\times [0, T)$. 
The relation \eqref{remove lip5} also enables us to deduce that 
\[
{|x_\vep- \lambda y_\vep|^2\over \vep}\leq u_\delta(x_\vep, t_\vep)-u_\delta(\lambda\ol{x}, \ol{t})-\lambda v(y_\vep, t_\vep)+\lambda v(\ol{x}, \ol{t})- {1 \over T-\ol{t}},
\]
which implies the existence of $(x_0, t_0)\in \ol{Q}_T$ such that along a subsequence
\[
x_\vep,  \lambda y_\vep\to x_0, \quad t_\vep\to t_0
\]
as $\vep\to 0$. Hereafter we still index the converging subsequence by $\vep$ for convenience of notation. 

By the semicontinuity of $u_\delta$ and $v$, we can use \eqref{remove lip5} to get 
\[
\Psi_{\delta, \lambda}\left({x_0\over \lambda}, t_0\right) \geq \limsup_{\vep\to 0}\Phi_{\lambda, \vep}(x_\vep, y_\vep, t_\vep)\geq \max_{\ol{Q}_T} \Psi_{\delta, \lambda} >0,
\]
which yields $t_0\geq 2C\sqrt{\delta}$ from \eqref{range lambda3} and thus $t_\vep>C\sqrt{\delta}$ when $\vep>0$ is sufficiently small. 


Moreover, we have $x_\vep\in (-b, b)$ when $\vep>0$ is sufficiently small since $x_\vep$ converges to $x_0 \in [-\lambda b, \lambda b]$ as $\vep \to 0$. In addition, since, by definition, $v(x, t)$ cannot be tested at $x=\pm b$, we have $y_\vep\in (-b, b)$. We fix such $\vep>0$.

We next apply the Crandall-Ishii lemma (cf. \cite[Theorem 8.3]{CIL}) to find, for any $\sigma>0$ small, 
$(\tau_1, p_1,  z_1)\in \ol{P}^{2, +} u_\delta(x_\vep, t_\vep)$ and $(\tau_2, p_2, z_2)\in \ol{P}^{2, -} v(y_\vep, t_\vep)$ satisfying 
\beq\label{semijet00}
 \tau_1={ \lambda \tau_2}+{1\over (T-t_\vep)^2},
\eeq
\beq\label{semijet1}
p_1=p_2={2(x_\vep-\lambda y_\vep)\over \vep}, 
\eeq
and 
\beq \label{semijet2}
\begin{pmatrix}
z_1 & 0\\
0 & -\lambda z_2
\end{pmatrix} \le
\dfrac{2}{\vep}\begin{pmatrix} 1 & -\lambda \\ -\lambda & \lambda^2 \end{pmatrix}+{4\sigma\over \vep^2}\begin{pmatrix} 1 & -\lambda \\ -\lambda & \lambda^2 \end{pmatrix}^2. 
\eeq
The time Lipschitz regularity of $u_\delta$ as in \eqref{time lip} yields 
\begin{equation}\label{bdd-jet-t}
|\tau_1|\leq L_\delta,
\end{equation}
and therefore by \eqref{range lambda1} and \eqref{semijet00} again
\beq\label{semijet0}
\tau_1-\tau_2\geq -(\lambda^{-1} - 1)L_\delta+{1\over \lambda(T-t_\vep)^2}\geq {1 \over 2T^2}.
\eeq
It follows from the monotonicity of $f^{-1}$, \eqref{range lambda2}, \eqref{bdd-jet-t} and \eqref{semijet0} that 
\begin{equation} \label{time contra}
\lambda f^{-1}(\tau_1)-f^{-1}(\tau_2) \ge \min_{|s| \le L_\delta} \left\{\lambda f^{-1}(s) - f^{-1}\left(s - \dfrac{1}{2T^2}\right)\right\} > 0. 
\end{equation}
Multiplying both sides of \eqref{semijet2} by the vector $(\lambda, 1)$ from left and from right, we deduce that 
\beq\label{semijet3}
\lambda z_1\leq z_2.
\eeq

We then adopt the definition of viscosity sub- and supersolutions for $u_\delta$ and $v$ respectively to deduce that 
\[
\tau_1+F(p_1, z_1)\leq 0
\]
and
\[
\tau_2+F(p_2, z_2)\geq 0,
\]
which yields by the assumption (A1), \eqref{semijet1} and \eqref{semijet3}
\[
\lambda f^{-1}(\tau_1)-f^{-1}(\tau_2) \leq \lambda g(p_1)z_1 - g(p_2)z_2 \le 0. 
\]
This is a contradiction to \eqref{time contra}. 

\ul{\textit{Step 4. Comparison away from the boundary}}

Case 2. Suppose that for any $\hat{t}>0$ neither $\hat{x}=b$ nor $\hat{x}=-b$ satisfies \eqref{contra cond}. 
In other words, we have 
\beq\label{contra cond2}
\sup_{(x, t)\in \{\pm b\}\times [0, T)} \left\{u_\delta(x, t)-v(x, t)-{1\over T-t}\right\}\leq 0.
\eeq
We essentially get a comparison relation like the Dirichlet boundary condition. 
In this case, for any $\vep>0$, we take the standard auxiliary function
\[
\Phi_{\vep}(x, y, t):=u_\delta(x, t)-v(y, t)-{|x- y|^2\over \vep}-{1 \over T-t}
\]
for $x, y\in \ol{Q}_T$ and $t \in [0, T)$. It is clear that $\Phi_\vep$ attains a maximum over $\ol{Q}_T^2$ at $(x_\vep', y_\vep', t_\vep')$ with 
\[
\Phi_\vep(x_\vep', y_\vep', t_\vep')\geq \max_{\ol{Q}_T} \Psi_{\delta}.
\]
 Then using a similar argument as shown in Case 1, we can take a subsequence such that 
$x_\vep', y_\vep'\to x_0$ and $t_\vep' \to t_0$ as $\vep\to 0$ with $(x_0, t_0)\in \ol{Q}_T$. 
It follows from the semicontinuity of $u$ and $v$ again that 
\[
u_\delta (x_0, t_0)-v(x_0, t_0)-{1 \over T-t_0}\geq \limsup_{\vep\to 0} \Phi_\vep(x_\vep, y_\vep, t_\vep)\geq  \max_{\ol{Q}_T} \Psi > 0,
\]
which implies that $(x_0, t_0)$ is a maximizer of $\Psi$ in $\ol{Q}_T$. Then by \eqref{remove_lipd} and \eqref{contra cond2}, we have $t_0\geq 2C\sqrt{\delta}$ and $x_0\in (-b, b)$.
Hence, $x_\vep', y_\vep'\in (-b, b)$ and $t_\vep'>C\sqrt{\delta}$ when $\vep>0$ is sufficiently small. We fix such $\vep>0$.

The argument below is almost the same as that in Case 1. We  apply the Crandall-Ishii lemma again to get $(\tau_1', p_1',  z_1')\in \ol{P}^{2, +} u_\delta(x_\vep', t_\vep')$ and $(\tau_2', p_2', z_2')\in \ol{P}^{2, -} v(y_\vep', s_\vep')$ for any $\sigma>0$ small such that
\beq\label{semijet0'}
\tau_1'-\tau_2'={1\over (T-t_\vep')^2}\geq {1\over T^2}, 
\eeq
\beq\label{semijet1'}
p_1'=p_2'={2(x_\vep'-y_\vep')\over \vep}, 
\eeq
and 
\beq \label{semijet2'}
\begin{pmatrix}
z_1' & 0\\
0 & - z_2'
\end{pmatrix}
\leq
\dfrac{2}{\vep}\begin{pmatrix} 1 & -1 \\ -1 & 1 \end{pmatrix}+{4\sigma\over \vep^2}\begin{pmatrix} 1 & -1 \\ -1 & 1 \end{pmatrix}^2.
\eeq
The last inequality implies that $z_1'\leq z_2'$. 

Since \eqref{time lip} implies that $|\tau_1'|\leq L_\delta$, by \eqref{semijet0'} together with the continuity and the monotonicity of $f^{-1}$, we have
\beq\label{contra cond3}
f^{-1}(\tau_1')-f^{-1}(\tau_2') \geq \min_{|s| \le L_\delta} \left\{f^{-1}(s) - f^{-1}\left(s - \dfrac{1}{T^2}\right) \right\} > 0.
\eeq
We next use the definition of sub- and supersolutions again to obtain 
\[
\tau_1'+ F(p_1', z_1')\leq 0, \quad \tau_2'+ F(p_2', z_2')\geq 0.
\]
Applying \eqref{semijet1'} and \eqref{semijet2'}, we are led to
\[
f^{-1}(\tau_1')-f^{-1}(\tau_2') \leq g(p_1')z_1'-g(p_2')z_2'\leq 0,
\]
which yields a contradiction to \eqref{contra cond3}. 
\end{proof}
We remark that the scaling technique employed in Step 3 of the proof above is inspired by \cite{So}. As an immediate consequence of Theorem \ref{thm cp}, 
it is clear that
\beq\label{initial dependence}
\sup_{\ol{Q}}|u-v| \leq \max_{[-b, b]} |u(\cdot, 0)-v(\cdot, 0)|
\eeq
if both $u$ and $v$ are continuous viscosity solutions of \eqref{flow eq}.  
In particular, there exists at most one continuous viscosity solution $u$ of \eqref{flow eq} and \eqref{initial}.

\section{Existence of solutions}\label{sec:existence}

In this section, we discuss existence of solutions of \eqref{flow eq}--\eqref{initial}, for which  the range of $\alpha$ in \eqref{exponent alpha} plays a key role.  
Under the assumptions (A1)--(A2), we first use Perron's method to show existence of viscosity solutions of \eqref{flow eq}--\eqref{initial} with a class of initial values.

We also study the case when a different range of decay exponent $\alpha$ is taken. Assuming $\alpha\leq 2$ as in (A2') instead of (A2), we show nonexistence of spatially bounded solutions of \eqref{flow eq}--\eqref{initial} by constructing rapidly evolving subsolutions. If the boundedness condition on the notion of solutions is removed, then one can see that any such solution must instantaneously blow up near the boundary.

\subsection{Existence for $\alpha> 2$}


In this section we give a proof of Theorem \ref{thm existence perron}. Let us first show the existence result when $u_0$ satisfies the stronger regularity condition (A3).

\begin{prop}[Existence of solutions with regular initial data]\label{prop existence truncated}
Let $b>0$. 
Assume that functions $f,g$ and $u_0$ satisfy (A1), (A2) and (A3), respectively.  
Then there exists a unique viscosity solution of \eqref{flow eq}--\eqref{initial}.
\end{prop}

\begin{proof}
For any $u_0$ satisfying \eqref{initial cond1} and \eqref{initial cond2}, it is not difficult to see that 
\beq\label{initial barriers}
v_\pm(x, t)=u_0(x) + L_\pm t\quad (x, t)\in \ol{Q}
\eeq
are respectively a super- and a subsolution of \eqref{flow eq} and \eqref{bdry}; concerning the boundary behavior of $v_+$, it satisfies \eqref{bdry}, since it cannot be tested from below by $C^2$ functions at $\pm b$. Moreover, $v_\pm$ are continuous in $\ol{Q}$ and satisfy \eqref{initial}. 

We now adopt Perron's method (cf. \cite{I1, CIL, ISa}) to show the existence of solutions of \eqref{flow eq} and \eqref{bdry}. Put, for any $(x, t)\in \ol{Q}$, 
\beq\label{perron def}
u(x, t)=\inf\left\{v(x, t):  v_-\leq v \leq v_+ \ \text{and $v$ is a supersolution of \eqref{flow eq} and \eqref{bdry}}\right\}.
\eeq
We take the upper and lower semicontinuous envelopes of $u$: 
\[
\begin{aligned}
&u^\ast(x, t)=\lim_{\delta\to 0+}\sup\left\{u(y, s): (y, s)\in \ol{Q}, \ |x-y|+|t-s|\leq \delta\right\},\\
&u_\ast(x, t)=\lim_{\delta\to 0+}\inf\left\{u(y, s): (y, s)\in \ol{Q}, \ |x-y|+|t-s|\leq \delta\right\}
\end{aligned}
\]
for all $(x, t)\in \ol{Q}$.

One can follow the classical arguments to verify the subsolution property of $u^\ast$ and the supersolution property of $u_\ast$  in the interior. 
Let us prove in detail that $u_\ast$ satisfies the boundary condition \eqref{bdry}.

Suppose by contradiction that there exist $t_0>0$ and $\phi\in C^2(\ol{Q})$ such that $u_\ast-\phi$ attains a minimum in $\ol{Q}$ at the point $(b, t_0)$. (The case for tests at the other endpoint can be similarly handled.) 
Let $\delta > 0$ be sufficiently small so that $u_\ast-\phi_\delta$
attains a strict minimum at $(b, t_0)$ over a neighborhood $\mathcal{N}$ in $\ol{Q}$, where
\beq\label{test perturb}
\phi_\delta(x, t)=\phi(x, t)+\delta(x-b)+{1\over \delta}(x-b)^2-\delta(t-t_0)^2
\eeq
for $(x, t)\in \ol{Q}$. 

By definition of $u_\ast$, we can find a sequence $(y_k, s_k)\in \ol{Q}$ such that $u(y_k, s_k)\to u_\ast (b, t_0)$ as $k\to \infty$. By \eqref{perron def}, we can further take a supersolution $v_k$ of \eqref{flow eq} and \eqref{bdry} for each $k$ such that 
\[
u(y_k, s_k)\leq v_k(y_k, s_k)\leq u(y_k, s_k)+{1\over k}.
\]
It follows immediately that $v_k(y_k, s_k)\to u_\ast(b, t_0)$ as $k\to \infty$. 

Let us take a minimizer $(x_k, t_k)$ of $v_k-\phi_\delta$ over a compact subset $\mathcal{K}$ of $\mathcal{N}$ containing $(b, t_0)$. 
It follows that $(x_k, t_k)\to (b, t_0)$ as $k\to \infty$, for otherwise a different limit would admit another minimizer of $u_\ast-\phi_\delta$ in $\mathcal{K}$, which contradicts the strict minimality in $\mathcal{N}$ at $(b, t_0)$.
Hence, $v_k-\phi_\delta$ attains a local minimum at $(x_k, t_k)$. 

Since $v_k$ satisfies the singular boundary condition \eqref{bdry}, we have $x_k< b$ for all $k$. 
We then apply the definition of supersolutions of \eqref{flow eq} for $v_k$ to get
\[
\phi_t(x_k, t_k)-2\delta(t_k-t_0)+F\left(\phi_x(x_k, t_k)+\delta+{2\over \delta}(x_k-b), \ \phi_{xx}(x_k, t_k)+{2\over \delta}\right)\geq 0.
\]
Sending $k\to \infty$, we have 
\[
\phi_t(b, t_0)+F\left(\phi_x(b, t_0)+\delta,\  \phi_{xx}(b, t_0)+{2\over \delta}\right)\geq 0.
\]
Letting $\delta>0$ small, we end up with a contradiction, since it follows from (A1) and (A2) that
\[
F\left(\phi_x(b, t_0)+\delta,\  \phi_{xx}(b, t_0)+{2\over \delta}\right)\to -\infty\quad \text{as $\delta\to 0$}.
\]
We have thus proved that $u_\ast$ is a supersolution of \eqref{flow eq} and \eqref{bdry}.

We also need to make sure that the boundary values of $u^\ast$ satisfy the accessibility condition \eqref{space conti}. To this purpose, we simply set
\beq\label{bdry value change}
\tilde{u}(x, t)=\begin{cases}
\displaystyle u^\ast(x, t) &\text{if $x\in (-b, b),\ t\geq 0$,}\\
\displaystyle \limsup_{\substack{(y, s)\to (x, t)\\ (y, s)\neq (x, t)}}u^\ast(y, s) &\text{if $x=\pm b,\ t\geq 0$.}
\end{cases}
\eeq
Then $\tilde{u}$ is sill upper semicontinuous in $\ol{Q}$ and is a subsolution of \eqref{flow eq} satisfying the accessibility condition. Also, we have $u_\ast\leq \tilde{u}$ in $\ol{Q}$. 

In view of \eqref{perron def}, it is not difficult to see that 
\[
\tilde{u}(\cdot, 0)=u_\ast(\cdot, 0)=u_0 \quad \text{in $[-b, b]$.}
\]
We thus can apply Theorem \ref{thm cp} to deduce that $\tilde{u}\leq u_\ast$ in $[-b, b]\times [0, \infty)$, which immediately yields the existence of a unique continuous solution of \eqref{flow eq} and \eqref{bdry} satisfying \eqref{initial}. 
\end{proof}

In order to obtain the existence of solutions for more general initial values, we establish a stability result for \eqref{flow eq}--\eqref{bdry} by applying the proof of Proposition \ref{prop existence truncated}. 
Let $u_\vep: \ol{Q}\to \R$ be a family of functions uniformly bounded in $\ol{Q}_T$ for every $T>0$. We define the relaxed half limits of $u_\vep$ as follows. 
For any $(x, t)\in \ol{Q}$, take 
\beq\label{stability limits}
\begin{aligned}
\ol{u}(x, t)&=\limsups_{\vep\to 0} u_\vep(x, t)=\lim_{\delta\to 0} \sup\{u_\vep(y, s): (y, s)\in \ol{Q},\ |x-y|+|t-s|+\vep\leq \delta\},\\
\ul{u}(x, t)&=\liminfs_{\vep\to 0} u_\vep(x, t)=\lim_{\delta\to 0} \inf\{u_\vep(y, s): (y, s)\in \ol{Q},\ |x-y|+|t-s|+\vep\leq \delta\}.
\end{aligned}
\eeq
\begin{thm}[Stability]\label{thm stability}
Let $u_\vep$ a viscosity solution of \eqref{flow eq} and \eqref{bdry} for each $\vep>0$. Assume that $u_\vep$ are uniformly bounded in $\ol{Q}_T$ for every $T>0$. 
Then $\ol{u}$ and $\ul{u}$ given by \eqref{stability limits} are respectively a sub- and a supersolution of \eqref{flow eq} satisfying \eqref{bdry}. In particular, if $u_\vep\to u$ converges uniformly in $\ol{Q}$, then $u\in C(\ol{Q})$ is a solution of \eqref{flow eq} and \eqref{bdry}. 
\end{thm}
\begin{proof}
We omit the verification of sub- and supersolution properties in $Q$, since it follows from the standard stability theory (cf. \cite{CIL}).  In order to show the supersolution property of $\ul{u}$ on the boundary, we can use the same argument in the proof of Proposition \ref{prop existence truncated}. \\
\mbox{}\\
Suppose by contradiction that there exists $\phi\in C^2(\ol{Q})$ such that $\ul{u}-\phi$  attains a minimum at $(b, t_0)$ for some $t_0>0$. Following the proof of Proposition \ref{prop existence truncated}, we can find a sequence $u_\vep$ and $(x_\vep, t_\vep)\in \ol{Q}$, still indexed by $\vep$ for simplicity, such that $(x_\vep, t_\vep)\to (x_0, t_0)$ as $\vep\to 0$ and 
$u_\vep-\phi_\delta$ attains a local maximum at $(x_\vep, t_\vep)$, where $\phi_\delta$ is given by \eqref{test perturb}. 
It is clear that $x_\vep\neq b$ due to the singular boundary condition. We thus have 
\[
\phi_t(b, t_\vep)+F\left(\phi_x(x_\vep, t_\vep)+\delta+{2\over \delta}(x_\vep-b), \phi_{xx}(x_\vep, t_\vep)+{2\over \delta}\right)\geq 0.
\]
Letting $\vep\to 0$ , we obtain 
\[
\phi_t(b, t_0)+F\left(\phi_x(b, t_0)+\delta, \phi_{xx}(b, t_0)+{2\over \delta}\right)\geq 0,
\]
which leads to a contradiction if $\delta>0$ is taken sufficiently small. 
\end{proof}

Another important ingredient for our proof of the existence theorem is the following result on regularization of $u_0$. 

\begin{lem}[Regularization of initial data]\label{lem regularization}
Let $u_0 \in C([-b,b])$ be a continuous function. 
Then, for any $\vep>0$ small, there exist $L_\pm(\vep) \in \mathbb{R}$ and $\psi_\vep \in C^\infty((-b, b)) \cap C([-b, b])$ such that
\begin{align}
&\sup_{x\in [-b, b]}|\psi_\vep(x)-u_0(x)|\to 0 \quad \text{as $\vep\to 0$, } \label{psi-1} \\
&\lim_{x \to \pm b} (\psi_\vep)_x(x) = \pm \infty, \label{psi-2} \\
&L_-(\vep) \le F((\psi_\vep)_x, (\psi_\vep)_{xx}) \le L_+(\vep).\label{psi-3}
\end{align}
Furthermore, if $u_0$ is convex, then $\psi_\vep$ is strictly convex and $L_-(\vep) > 0$ for any $\vep>0$ small. 
\end{lem}

\begin{proof}
We first argue without assuming the convexity of $u_0$. 
Let $\hat{\psi}(x)\in C([-b,b]) \cap C^\infty((-b,b))$ be a strictly convex function such that 
\[
\hat{\psi}(x)=-(|b|-|x|)^{\alpha-2\over \alpha-1}\qquad \text{for all $x\in [-b, -b+r]\cup [b-r, b]$}
\]
One can connect these two pieces smoothly with a smooth strictly convex graph to construct the entire $\hat{\psi}$ in $[-b, b]$. 
For any $\vep>0$, set
\[
\psi_{\vep, 1}=\vep\hat{\psi}\quad \text{in $[-b, b]$}.
\]
Then $\psi_{\vep,1}\to 0$ uniformly on $[-b,b]$ as $\vep\to 0$ and, for each fixed $\vep>0$,  
\begin{equation}\label{approxi-psi-1}
\lim_{x \to \pm b} (\psi_{\vep,1})_x(x) =\pm \infty
\end{equation}
and we have by the assumption (A2) 
\begin{equation} \label{approxi-psi-I}
\begin{aligned}
\lim_{x \to b} F\left((\psi_{\vep,1})_x(x), (\psi_{\vep,1})_{xx}(x)\right) =&\; \lim_{x \to b} f\left(g\left(\frac{\vep(\alpha-2)}{\alpha-1} (b-x)^{-\frac{1}{\alpha-1}}\right)\frac{\vep(\alpha-2)}{(\alpha-1)^2}(b-x)^{-\frac{\alpha}{\alpha-1}}\right) \\
=&\; f\left(\dfrac{\vep^{1-\alpha}(\alpha-2)^{1-\alpha}}{(\alpha-1)^{2-\alpha}} C_+\right)>0. 
\end{aligned}\end{equation}
Similarly, at the other boundary point $x=-b$,  we have 
\begin{equation} \label{approxi-psi-II}
\lim_{x \to -b} F\left((\psi_{\vep,1})_x(x),(\psi_{\vep,1})_{xx}(x)\right) = f\left(\dfrac{\vep^{1-\alpha}(\alpha-2)^{1-\alpha}}{(\alpha-1)^{2-\alpha}} C_-\right)>0. 
\end{equation}
Fix a nonnegative function $\rho \in C^\infty_c(\mathbb{R})$ with support contained in $[-1, 1]$ such that $\int_\mathbb{R} \rho \; dx = 1$. We next take a closed interval $I_\vep=[-b/(1-\vep), b/(1-\vep)]$ and
define $u_{0,\vep} \in C(I_\vep)$ and $\psi_{\vep, 2} \in C^\infty([-b,b])$ by 
\[ \begin{aligned}
&u_{0,\vep}(x) := u_0((1-\vep)x) \quad \text{for} \; \; x \in I_\vep, \\
&\psi_{\vep, 2}(x) := {1\over \sigma(\vep)}\int_{I_\vep} u_{0,\vep}(y) \cdot  \rho\left(\frac{x-y}{\sigma(\vep)}\right) \; dy \quad \text{for} \; \; x \in [-b,b], 
\end{aligned}\]
where we let $\sigma(\vep)<\vep b/(1-\vep)$ for any $\vep > 0$ small.
Then, by a standard argument for mollification, the continuity of $u_0$ and the definition of $u_{0,\vep}$ imply that $\psi_{\vep, 2}\to u_0$ uniformly on $[-b,b]$ as $\vep \to 0$. 
Thus, letting 
\[
 \psi_\vep(x) := \psi_{\vep,1}(x) + \psi_{\vep,2}(x) \quad \text{for} \; \; x \in [-b,b],
\]
we obtain \eqref{psi-1}. 
Furthermore, since $|(\psi_{\vep,2})_x|$ is bounded in $[-b, b]$, \eqref{approxi-psi-1} implies \eqref{psi-2}. 
Also, since $\psi_{\vep, 2} \in C^\infty([-b,b])$, 
the blow-up rates of $(\psi_\vep)_x$ and $(\psi_\vep)_{xx}$ at the boundary points $x=\pm b$ are the same as those of $(\psi_{\vep,1})_x$ and $(\psi_{\vep,1})_{xx}$ respectively. It follows from 
\eqref{approxi-psi-I} and \eqref{approxi-psi-II} that
\begin{equation}\label{approxi-psi-3}
0<\lim_{x \to \pm b} F\left((\psi_\vep)_x(x), (\psi_\vep)_{xx}(x)\right) = \lim_{x \to \pm b} F\left((\psi_{\vep,1})_x(x), (\psi_{\vep,1})_{xx}(x)\right) < \infty. 
\end{equation}
We therefore can find $L_\pm(\vep)\in \R$ such that  \eqref{psi-3} holds. 

We next discuss the special case when $u_0$ is convex. 
In this case, we can see again by a standard argument of mollification that $\psi_{\vep,2}$ is convex due to the convexity of $u_{0, \vep}$ in $I_\vep$. 
Since $\psi_{\vep, 1}$ is a strictly convex in $[-b, b]$, so is $\psi_\vep$. 
We are thus led to \eqref{psi-3} with $L_-(\vep)>0$ because of  \eqref{approxi-psi-3}. 
\end{proof}

Let us now prove Theorem \ref{thm existence perron}. 

\begin{proof}[Proof of Theorem \ref{thm existence perron}]
Let $\psi_\vep$ be given as in Lemma \ref{lem regularization}. 
Using Proposition \ref{prop existence truncated}, for any $\vep>0$ small we can find a solution $u_\vep$ of \eqref{flow eq} and \eqref{bdry} with initial value $\psi_\vep$. Note that \eqref{psi-1} implies the existence of $\delta(\vep)>0$ with $\delta(\vep)\to 0$ as $\vep\to 0$ such that
\[
\max_{x\in [-b, b]}|\psi_\vep(x)-u_0(x)|\leq \delta(\vep).
\] 
We thus have 
\[
\max_{x\in [-b, b]}|\psi_{\vep}(x)-\psi_{\sigma}(x)|\leq \delta(\vep)+\delta(\sigma)
\]
for any $\vep, \sigma>0$ small. 
For any fixed $\sigma>0$ and any $\vep>0$, one can also use \eqref{psi-1}--\eqref{psi-3} to show that
\[
(x, t)\mapsto \psi_{\sigma}(x)+\delta(\sigma) +\delta(\vep)+ L_+(\sigma)t
\]
and
\[
(x, t)\mapsto \psi_\sigma(x)-\delta(\sigma) -\delta(\vep)+L_-(\sigma)t
\]
are respectively a supersolution and a subsolution of \eqref{flow eq}--\eqref{bdry}. In view of Theorem \ref{thm cp} and Proposition \ref{prop existence truncated}, there exists a solution $u_\vep$ of \eqref{flow eq}--\eqref{bdry} with 
\[
u_\vep(\cdot, 0)=\psi_\vep\quad \text{ in $[-b, b]$}
\] 
such that  
\beq\label{eq existence revise1}
\psi_\sigma(x)-\delta(\sigma)-\delta(\vep)+L_-(\sigma) t\leq u_\vep(x, t)\leq \psi_\sigma(x)+\delta(\sigma)+\delta(\vep)+L_+(\sigma) t
\eeq
for all $[x, t]\in \ol{Q}$ and $\vep<\sigma$. In particular, $u_\vep$ are uniformly bounded in $\ol{Q}_T$ for each $T>0$. 

By Theorem \ref{thm stability}, we deduce that the relaxed half limits $\ol{u}$ and $\ul{u}$ given by \eqref{stability limits}
are respectively a sub- and a supersolution of \eqref{flow eq}--\eqref{bdry}. 
We also define $\tilde{u}$ by changing the boundary values of $\ol{u}$ in the same way as in \eqref{bdry value change}. 
Then $\tilde{u}$ is still a subsolution of \eqref{flow eq} and satisfies the accessibility condition in Theorem \ref{thm cp}.
In addition, it follows from \eqref{eq existence revise1}, the upper semicontinuity of $\ol{u}$ and the continuity of $\psi_\sigma$ that 
\[
\psi_\sigma-\delta(\sigma)+L_-(\sigma) t\leq\ul{u}\leq \tilde{u}\leq \ol{u} \leq \psi_\sigma+\delta(\sigma)+L_+(\sigma) t\quad \text{ in $\ol{Q}$}
\]
which in particular yields 
\[
\psi_\sigma-\delta(\sigma)\leq \ul{u}(\cdot , 0)\leq \tilde{u}(\cdot,0) \leq \ol{u}(\cdot , 0)\leq \psi_\sigma+\delta(\sigma) \quad \text{ in $[-b, b]$}.
\]
Sending $\sigma\to 0$, we then obtain 
\[
\ul{u}(\cdot , 0)= \tilde{u}(\cdot,0) = \ol{u}(\cdot , 0)=u_0 \quad \text{in $[-b, b].$}
\]
We thus can adopt Theorem \ref{thm cp} to conclude that $\tilde{u}=\ul{u}$ is the desired unique solution of \eqref{flow eq}--\eqref{initial}.
\end{proof}

Let us 
 present several results that will be used in our asymptotic analysis in Section~\ref{sec:asymptotics}. 
We first show that $u$ obtained in Proposition \ref{prop existence truncated} is Lipschitz in time. 

\begin{prop}[Time Lipschitz continuity for regular initial data]\label{prop lip perron}
Let $b>0$. 
Assume that functions $f,g$ and $u_0$ satisfies (A1), (A2) and (A3). 
Let $u$ be the unique solution of \eqref{flow eq}, \eqref{bdry} and \eqref{initial} obtained in Proposition \ref{prop existence truncated}. 
Then 
\beq\label{time lip perron}
L_- (t-s) \le u(x, t)-u(x, s) \leq L_+(t-s)
\eeq
for all $x\in [-b, b]$ and $0 \le s < t < \infty$.
\end{prop}
\begin{proof}
Recalling that $v_+$ given in \eqref{initial barriers} is a supersolution of \eqref{flow eq} and \eqref{bdry}, by Theorem \ref{thm cp} we get
\[
u(x, h)\leq v_+(x, h)=u_0(x)+L_+ h
\]
for all $(x, h)\in \ol{Q}$. Since $(x, t)\mapsto u(x, t)+L_+h$ and $(x, t)\mapsto u(x, t+h)$ are two solutions of \eqref{flow eq} and \eqref{bdry}, we can apply Theorem \ref{thm cp} again to deduce that 
\[
u(x, t+h)\leq u(x, t)+L_+h
\]
for all $(x, t)\in \ol{Q}$ and $h\geq 0$. Using the lower barrier $v_-$, we can obtain 
\[
u(x, t+h)\geq u(x, t)+L_-h. 
\]
for all $(x, t)\in \ol{Q}$ and $h\geq 0$, which completes the proof of \eqref{time lip perron}.
\end{proof}

We next give an interpretation of Proposition~\ref{prop lip perron} in terms of the boundedness of $F(u_x, u_{xx})$ in the viscosity sense. 

\begin{cor}[Preservation of operator boundedness]\label{cor:fixt}
Let $b>0$. 
Assume that functions $f,g$ and $u_0$ satisfy (A1), (A2) and (A3), respectively. 
Let $u$ be the unique solution of \eqref{flow eq}--\eqref{initial}. 
Then, $u(\cdot,t)$ satisfy
\begin{equation}
-L_+ \le F(u_x, u_{xx}) \le -L_- \quad \text{in $(-b, b)$} \label{fixt eq}
\end{equation}
in the viscosity sense for any $t > 0$. 
\end{cor}

\begin{proof}
Fix an arbitrary $s>0$. 
First, we prove that $u(\cdot,s)$ satisgy the second inequality in \eqref{fixt eq} in $(-b,b)$. 
Let $\phi \in C^2([-b,b])$ and $u(\cdot,s)-\phi$ attains a strict maximum in $[-b,b]$ at $x_0 \in (-b,b)$. 
Then there exists a maximum point $(x_\varepsilon, t_\varepsilon) \in \ol{Q}$ for 
\[ \Phi_\varepsilon(x,t) := u(x,t) - \phi(x) - \dfrac{|t-s|^2}{\varepsilon} \]
for $\varepsilon > 0$ by virtue of \eqref{time lip perron} and the uniform boundedness of $u$ in finite time intervals. 
Applying \eqref{time lip perron} again, we have 
\begin{equation}\label{fixt1} 
L_- \le \dfrac{2(t_\varepsilon - s)}{\varepsilon} \le L_+, 
\end{equation}
hence $t_\varepsilon \to s$ as $\varepsilon \to 0$. 
Since the Lipschitz continuity in time implies that $u(\cdot,t_\varepsilon)$ uniformly converges to $u(\cdot,s)$, we have also $x_\varepsilon \to x_0$. 
Thus, $x_\varepsilon \in (-b,b)$ and $t_\varepsilon > 0$ for sufficiently small $\varepsilon$ and we obtain by the definition of the subsolution for \eqref{flow eq} 
\[ \dfrac{2(t_\varepsilon - s)}{\varepsilon} + F(\phi_x(x_\varepsilon), \phi_{xx}(x_\varepsilon)) \le 0. \]
Applying \eqref{fixt1} and letting $\varepsilon \to 0$, we can see that $u(\cdot,s)$ is a solution of the second inequality in \eqref{fixt eq} in $(-b,b)$. 
For the first inequality in \eqref{fixt eq}, a similar argument can be applied. 
\end{proof}

Corollary \ref{cor:fixt} immediately yields the convexity of $u(\cdot, t)$ for any $t\geq 0$ if it is known that $L_- \ge 0$. 
We thus can obtain the following convexity preserving property by using the approximation argument in the proof of Theorem \ref{thm existence perron} together with the construction of $\psi_\vep$ in Lemma \ref{lem regularization} for convex $u_0$. 

\begin{cor}[Convexity preserving]\label{cor:convex}
Let $b>0$. 
Assume that functions $f, g$ satisfy (A1) and (A2). Let $u_0\in C([-b, b])$ be a convex function. 
Let $u$ be the unique solution of \eqref{flow eq}--\eqref{initial}. 
Then $u(\cdot, t)$ is convex for all $t\geq 0$. 
\end{cor}

\subsection{Nonexistence of spatially bounded solutions for $\alpha\leq 2$}\label{sec:nonexistence}

We consider the behavior of solution of \eqref{flow eq} and \eqref{bdry} with $g(p) \ge O(|p|^{-2})$ as $p \to \pm\infty$. 
We construct subsolutions $\ul{u} (\cdot;k)$ of \eqref{flow eq} and \eqref{bdry} that start from uniformly bound initial data but tend to infinity on the boundary at any positive time as $k \to \infty$. 

Let $b>0$. 
Assume $f$ and $g$ satisfy (A1) and (A2'). 
Then, there exist $p_0 > 0$ and $M > 0$ such that 
\begin{equation}\label{order-g-blowup} 
g(p) \ge 2M |p|^{-2} \quad \text{for} \; \; |p| \ge p_0. 
\end{equation}
For sufficiently large $k>0$, let 
\begin{equation}\label{choice-r} 
r_k := \sqrt{\frac{1+k^2}{k^2}} b 
\end{equation}
and $y_k \in (0, e^{-1})$ is chosen to satisfy 
\begin{equation}\label{choice-y}
\frac{1}{y_k\log y_k} = -k. 
\end{equation}
Define $x(k,t)$ by 
\begin{equation}\label{choice-x} 
x(k,t) = \exp\left(-\exp \{f(-M\log y_k) t + \log(-\log y_k)\}\right) + b - y_k \quad \text{for} \; \; t \ge 0. 
\end{equation}
Define 
\begin{equation}\label{def:lower-blowup}
 \ul{u} (x,t;k) := 
\begin{cases}
\displaystyle \log\left(\frac{\log\{x(k,t) - x + y_k\}}{\log y_k}\right) - \sqrt{r_k^2 - x(k,t)^2} & \text{for} \; \; x \in (x(k,t), b], \\
\displaystyle -\sqrt{r_k^2 - x^2} & \text{for} \; \; x \in [-x(k,t), x(k,t)], \\
\displaystyle \log\left(\frac{\log\{x(k,t) + x + y_k\}}{\log y_k}\right) - \sqrt{r_k^2 - x(k,t)^2} & \text{for} \; \; x \in [-b, -x(k,t)). 
\end{cases} 
\end{equation}

\begin{rmk}
Let us make some remarks on our choice of $r_k$, $y_k$ and $x(k, t)$ above. 
\begin{itemize}
\item $r_k$ in \eqref{choice-r} is chosen so that the derivative of the graph of the arc with radius $r_k$ centered at the origin in $\R^2$ is $\pm k$ at the boundary $x=\pm b$. 
\item The choice of $y_k$ in \eqref{choice-y} is unique in $(0,e^{-1})$ for sufficiently large $k>0$ and $y_k \to 0$ as $k \to \infty$. 
\item $x(k,t)$ in \eqref{choice-x} is the solution of the ordinary differential equation 
\begin{equation}\label{ODE-x} 
\frac{x'(t)}{\{x(t)-b+ y_k\} \log\{x(t)-b + y_k\}} = f(-M \log y_k) 
\end{equation}
in $[0, \infty)$ for any $k>0$ large and $x(0) = b$. 
\end{itemize}
\end{rmk}

\begin{lem}[Construction of subsolutions]\label{lem:subsolution-blowup}
Let $b > 0$. 
Assume that functions $f$ and $g$ satisfy (A1) and (A2'). 
Then, $\ul{u}(\cdot, k)$ defined by \eqref{def:lower-blowup} is continuous in $\ol{Q}$ and a subsolution of \eqref{flow eq} and \eqref{bdry} for sufficiently large $k>0$. 
\end{lem}

\begin{proof}
For simplicity of notation, below we omit the parameter $k$ in $\ul{u}$.
We can easily see that $\ul{u}$ is continuous in $\ol{Q}$ and smooth in $Q \setminus \Gamma_k$, where  
\[
\Gamma_k=\cup_{t > 0}\{(\pm x(k,t),t)\}.
\] 
Furthermore, the choice of $y_k$ \eqref{choice-y} yields 
\begin{equation}\label{order-ux} \begin{aligned}
&\lim_{x \to x(k,t) + } \ul{u}_x (x,t) = k > \lim_{x \to x(k,t) -} \ul{u}_x(x,t), \\
&\lim_{x \to -x(k,t) - } \ul{u}_x (x,t) = -k < \lim_{x \to -x(k,t) +} \ul{u}_x(x,t)  
\end{aligned} \end{equation}
for $t > 0$. 
Thus, $\ul{u}(x,t) - \phi(x,t)$ cannot attain a local maximum at $(\pm x(k,t), t)$ for any $t > 0$ and $\phi \in C^2(\ol{Q})$. 
Indeed, if we assume $\ul{u}(x,t) - \phi(x,t)$ attains a local maximum at $(x(k,t_0), t_0)$ for some $t_0 > 0$ and $\phi \in C^2(\ol{Q})$, we have 
\[ \ul{u}(x(k,t_0)+h, t_0) - \ul{u}(x(k,t_0), t_0) \le \phi(x(k,t_0)+h, t_0) - \phi(x(k,t_0), t_0) \quad \text{for} \; \; h \in \R, \]
which yields 
\[ \lim_{x \to x(k,t_0) +}\ul{u}_x(x,t_0) \le \phi_x(x(k,t_0),t_0) \le \lim_{x \to x(k,t_0) -} \ul{u}_x(x,t_0). \]
This contradicts \eqref{order-ux}. 
When $\ul{u}(x,t) - \phi(x,t)$ attains a local maximum at $(-x(k,t_0), t_0)$, we also obtain a contradiction by a similar argument. 
Therefore, it is sufficient to prove that $\ul{u}$ satisfies 
\[ \ul{u}_t \le f(g(\ul{u}_x) \ul{u}_{xx}) \quad \text{in} \; \; Q \setminus \Gamma_k\]
in the classical sense. 

Case 1. 
For $(x,t) \in Q \setminus \Gamma_k$ with $-x(k,t) < x < x(k,t)$, it is easily seen that $\ul{u}_t =0$ and $\ul{u}_{xx} > 0$. 
Thus, we have $\ul{u}_t \le f(g(\ul{u}_x) \ul{u}_{xx})$ by the conditions (A1) and (A2'). 

Case 2. 
For $(x,t) \in Q \setminus \Gamma_k$ with $x(k,t) < x < b$, we have 
\[ \ul{u}_x(x,t) = - \frac{1}{\{x(k,t) - x + y_k\}\log\{x(k,t) - x + y_k\}} \ge - \frac{1}{y_k\log y_k} = k \]
for sufficiently large $k > 0$. 
It implies by \eqref{order-g-blowup} and $y_k \to 0$ as $k \to 0$ that 
\begin{equation}\label{sub-blowup1}
\begin{aligned}
f(g(\ul{u}_x(x,t)) \ul{u}_{xx}(x,t)) \ge&\; f(2M(\ul{u}_x(x,t))^{-2} \ul{u}_{xx}(x,t)) \\
=&\; f(-2M(1 + \log\{x(k,t) - x + y_k\})) \\
\ge&\; f(-M\log\{x(k,t) - x + y_k\}) \\
\ge&\; f(-M\log y_k)
\end{aligned}
\end{equation}
for sufficiently large $k > 0$. 
On the other hand, we can also obtain by $x_t(k,t) <0$ and $x(k,t) > 0$ 
\begin{equation}\label{sub-blowup2}
\begin{aligned}
\ul{u}_t =&\; \frac{x_t(k,t)}{\{x(k,t) - x + y_k\}\log\{x(k,t) - x + y_k\}} + \frac{x(k,t) x_t(k,t)}{\sqrt{r_k^2 - x(k, t)^2}} \\
\le&\; \frac{x_t(k,t)}{\{x(k,t) - b + y_k\}\log\{x(k,t) - b + y_k\}}. 
\end{aligned}
\end{equation}
Since $x(k,t)$ is chosen to satisfy \eqref{ODE-x}, it follows from \eqref{sub-blowup1} and \eqref{sub-blowup2} that $\ul{u}_t \le f(g(\ul{u}_x) \ul{u}_{xx})$ holds. 
In the case $-b < x < -x(k,t)$, the inequality can be obtained similarly. 
\end{proof}

We are now in a position to prove Theorem \ref{thm:blowup}. 

\begin{proof}[Proof of Theorem \ref{thm:blowup}] 
Assume that $u\in C(\ol{Q})$ is a solution of \eqref{flow eq}--\eqref{initial} and fix $t_0 > 0$. 
For sufficiently large $k$, let $\ul{u}$ be the function defined by \eqref{def:lower-blowup}. 
Then, 
\[ u(x,0) \ge \min_{-b \le y \le b} u(y,0) \ge \ul{u}(x,t;k) + \min_{-b \le y \le b} u(y, 0) \quad \text{for} \quad x \in [-b,b]. 
 \]
Therefore, we obtain by the comparison principle 
\begin{equation}\label{blowup-ine-con}\begin{aligned} 
u(b, t_0) \ge&\; \ul{u}(b,t_0; k) + \min_{-b \le x \le b} u(x,0) \\
=&\; t_0 f(-M\log y_k) - \sqrt{r_k^2-x(k,t_0)^2} + \min_{-b \le x \le b} u(x,0) \\
\ge&\; t_0 f(-M\log y_k) - r_k + \min_{-b \le x \le b} u(x,0)
\end{aligned}\end{equation}
for sufficiently large $k>0$. 
Since $y_k \to 0$ and $r_k \to b$ as $k \to \infty$, we have $u(b,t_0) = \infty$ by letting $k \to \infty$ in \eqref{blowup-ine-con}, which contradicts the boundedness of $u$ on $[-b,b] \times [0,t_0]$. 
\end{proof}

\section{Stationary problem}\label{sec:stationary}

In this section, we first prove Theorem \ref{thm:ex-tw}. 
As we mentioned in the introduction, we will need the regularity theory for the profile function. 

\begin{lem}[Regularity improvement]\label{lem:regularity-TW} 
Let $b>0$. 
Assume that functions $f$ and $g$ satisfy (A1) and (A2). 
Assume also that $W \in C([-b,b]) \cap C^{0,1}((-b,b))$ and $c \in \mathbb{R}$ satisfy \eqref{traveling eq} in the viscosity sense. 
Then, $W \in C^{1,1}(-b,b)$. 
\end{lem}

\begin{proof}
Let $b_0 \in (0, b)$. 
Then, from $W \in C^{0,1}((-b,b))$, there exists $M>0$ such that 
\[ |W(x) - W(y)| \le M |x-y| \quad \text{for} \quad x,y \in [-b_0, b_0]. \]
Thus, $W$ has derivative $W_x$ almost everywhere in $[-b_0, b_0]$ and $|W_x(x)| \le M$ if the derivative exists. 
It yields that $W$ satisfies 
\begin{equation}\label{ine-Wxx}
\begin{aligned}
A_1 := f^{-1}(c)/ \Big(\max_{|p| \le M} g(p)\Big) \le&\; f^{-1}(v)/g(W_x) = W_{xx}  \\
\le&\; f^{-1}(c)/\Big(\min_{|p| \le M} g(p)\Big) =: A_2 \quad \text{in $(-b_0, b_0)$}
\end{aligned}
\end{equation}
in the sense of viscosity. 
Then, $W(x)- A_1 x^2/2$ is convex and $W(x)- A_2x^2/2$ is concave in $[-b_0, b_0]$; in other words, $W$ is both semiconvex and semiconcave. 
Thus, it is clear that $W \in C^{1,1}((-b,b))$; see also Corollary \ref{cor:regular-convex}. 
\end{proof}

Let us now prove Theorem \ref{thm:ex-tw}. 

\begin{proof}[Proof of Theorem \ref{thm:ex-tw}]
By the assumption (A2), we can take
\begin{equation}\label{function-G}
G(p)=\int_{-\infty}^p g(s)\, ds
\end{equation} 
for all $p\in \R$. 
We take $c >0$ satisfying 
\begin{equation}\label{condition-c}
2bf^{-1}(c)=G(\infty)=\int_{-\infty}^\infty f(s)\, ds.
\end{equation}
Noticing that $G$ is strictly increasing in $\R$, we get its inverse function $G^{-1}: (0, 2bf^{-1}(c))\to \R$. 
Note that $G^{-1}$ is integrable near the endpoints due to the condition (A1); in fact, by using the change of variable $y=G(p)$, we get
\[
\int_{0}^{G(\infty)} G^{-1}(y)\, dy=\int_{-\infty}^\infty p G'(p)\, dp =\int_{-\infty}^\infty pg(p)\, dp<\infty.
\]
We therefore can take $W$ as follows:
\[
W(x)=\int_0^x G^{-1}(f^{-1}(c)(y+b))\, dy, \quad x\in [-b, b].
\]
One can easily verify that $W$ is of class $C([-b,b]) \cap C^2((-b,b))$ 
and $(W,c)$ satisfies 
\begin{equation}\label{profile-W2}
G(W_x(x))=f^{-1}(c) (x+b), \quad x\in (-b, b)
\end{equation}
with $W_x(\pm b)=\pm \infty$, which yields 
\begin{equation}\label{profile-W1}
G(W_x)_x=f^{-1}(c), \quad x \in (-b,b). 
\end{equation}
Thus, $(W,c)$ fulfill \eqref{traveling eq} and \eqref{traveling bdry} in the classical sense. 

Next, we assume that $(\tilde{W},\tilde{c}) \in \{C([-b,b]) \cap C^{0,1}((-b,b))\} \times \mathbb{R}$ is another solution pair of \eqref{traveling eq} and \eqref{traveling bdry} in the viscosity sense. 
By Lemma \ref{lem:regularity-TW}, we see that $\tilde{W}\in C^{1,1}((-b,b))$ satisfies 
\begin{equation}\label{profile-Wtilde}
G(\tilde{W}_x)_x=f^{-1}(\tilde{c}) \quad \text{a.e.\ in $(-b, b)$.}
\end{equation}
Note that the fundamental theorem of calculus holds for $G(\tilde{W}_x)$ since the Lipschitz regularity of $G(\tilde{W}_x)$ follows from that of $\tilde{W}_x$ and $G$. 
Thus, by the boundary condition \eqref{traveling bdry}, the speeds $\tilde{c}$ also should satisfy \eqref{condition-c}, which yields $\tilde{c}=c$. 
Furthermore, taking the difference between \eqref{profile-W1} and \eqref{profile-Wtilde}  yields 
\[
\{G(W_x) - G(\tilde{W}_x)\}_x = 0 \quad \text{a.e.\ in $(-b,b)$.} 
\]
Thus, $W = \tilde{W} + m$ for some $m \in \mathbb{R}$. 

Let us prove the H\"{o}lder continuity of $W$ in (ii). 
Hereafter, the power $a^\beta$ for any $a,\beta \in \mathbb{R}$ should be understood as the signed power $|a|^{\beta -1}a$. 
By the condition (A2), there exist constants $M_1,M_2 > 0$ such that 
\[
 M_1 |s|^{-\alpha} \le g(s) \le M_2 |s|^{-\alpha} \] 
for all $s\in \R$ with $|s|$ sufficiently large. 
For any $x \in (-b,b)$ sufficiently close to $-b$, let 
\[ s(x) := \left(\dfrac{\alpha-1}{M_1} f^{-1}(c) (x+b)\right)^{\frac{1}{1-\alpha}}. \]
Then, we obtain by a simple calculation 
\[ G(s(x)) = \int_{-\infty}^{-s(x)} g(p) \, dp \ge \int_{-\infty}^{-s(x)} - M_1 p^{-\alpha} \, dp = f^{-1}(c) (x+b), \]
which implies by using the change of variable $f^{-1}(c)(y+b)=G(p)$ 
\[\begin{aligned} 
|W(-b) - W(x)| =&\; -\int_{-b}^x G^{-1}(f^{-1}(c)(y+b)) \, dy \le \dfrac{-1}{f^{-1}(c)} \int_{-\infty}^{-s(x)} p G'(p) \, dp \\
\le&\; \dfrac{M_2}{f^{-1}(c)} \int_{-\infty}^{-s(x)} -p^{1 - \alpha} \, dp \le M (x+b)^{\frac{\alpha-2}{\alpha-1}} 
\end{aligned}\]
for some $M > 0$ if $x$ is sufficiently close to $-b$. 
At the other boundary point, $|W(b) - W(x)| \le \tilde{M} (b-x)^{(\alpha-2)/(\alpha-1)}$ for some $\tilde{M} >0$ can be proved similarly if $x$ is sufficiently close to $b$. 
The H\"{o}lder continuity at the boundary points and $C^2$-regularity in $(-b,b)$ yields $W \in C^{\frac{\alpha-2}{\alpha-1}}([-b,b])$. 
\end{proof}


\begin{thm}[Uniqueness of profiles]\label{thm:profile-shift}
Let $c>0$ and $W$ be the speed and the profile function of the traveling wave obtained in Theorem \ref{thm:ex-tw}. 
Assume that $\ol{W}, \ul{W} \in C([-b,b]) \cap C^{0,1}((-b,b))$ are respectively a subsolution and a supersolution of \eqref{traveling eq} and \eqref{traveling bdry} in the viscosity sense.  
Then $\ol{W}$ and $\ul{W}$ are differentiable, and satisfy $\ol{W}_x=\ul{W}_x=W_x$ in $(-b, b)$. 
In particular, $\ol{W}$ and $\ul{W}$ are of class $C^2((-b,b))$ and coincide with $W$ in $(-b, b)$ up to constants. 
\end{thm}

\begin{proof}
We first show that $\ul{W}_x=W_x$ almost everywhere in $(-b, b)$.
The Lipschitz regularity of $\ul{W}$ implies that $\ul{W}$ is differentiable almost everywhere, and the singularity boundary condition \eqref{traveling bdry} yields
\begin{equation}\label{eq profile-shift3}
\ul{W}_x(x_n)\to \pm \infty 
\end{equation}
for any sequence $\{x_n\}\subset (-b, b)$ such that $\ul{W}$ is differentiable at all $x_n$ and $x_n\to \pm b$ as $n\to \infty$. 
Also, for any $b_0\in (0, b)$, $\ul{W}$ is semiconcave in $(-b_0, b_0)$; in fact, we have $\ul{W}_{xx} \le A$ in the viscosity sense for some $A \in \mathbb{R}$ by a similar argument for \eqref{ine-Wxx}. 
Therefore $\ul{W}(x)-A x^2/2$ is concave in $[-b_0, b_0]$. 

Recall that $G$ is given by \eqref{function-G}. 
By Lemma \ref{lem:mono-integral}, $G(\ul{W}_x)$ is differentiable almost everywhere and it implies that
\[
\left(G(\ul{W}_x)\right)_x\leq f^{-1}(c) \quad a.e.\  \text{in $(-b_0, b_0)$}
\]
since $\ul{W}$ is a supersolution of \eqref{traveling eq}. 
Apply this inequality and \eqref{mono-integral} to get
\begin{equation}\label{eq profile-shift2}
G(\ul{W}_x(y))-G(\ul{W}_x(x))\leq \int_x^y \left(G(\ul{W}_x)\right)_x(s) \, ds \leq f^{-1}(c) (y-x)
\end{equation}
for all $-b < x < y <b$ where $\ul{W}$ is differentiable.

Let us now compare the behavior of $W$ and $\ul{W}$. 
Taking $x \to -b+$ in the set of differentiable points of $\ul{W}$ and applying \eqref{profile-W2} and \eqref{eq profile-shift3}, we have $G(\ul{W}_x(y)) \le G(W_x(y))$ at any differentiable point $y \in (-b,b)$ of $\ul{W}$. 
Thus, to prove that $\ul{W}_x=W_x$ in $(-b, b)$, we assume by contradiction that $\ul{W}_x(x_0)<W_x(x_0)$ for some $x_0\in (-b, b)$.
Noticing that $W$ is smooth and convex, we can find $\bar{x}<x_0$ such that $\ul{W}_x(x_0)=W_x(\bar{x})$. 
By \eqref{eq profile-shift2} and \eqref{profile-W2}, we thus have
\[ \begin{aligned}
G(\ul{W}_x(y)) & = (G(\ul{W}_x(y) - G(\ul{W}_x(x_0))) + G(\ul{W}_x(x_0)) \\
& \le  f^{-1}(c) (y-x_0) + f^{-1}(c)(\bar{x} + b) \\
& = f^{-1}(c)((y+\bar{x}-x_0)+b) = G(W_x(y+\bar{x}-x_0))
\end{aligned}\]
for almost every $y\in (x_0, b)$. This upper bound for $\ul{W}_x$ in $(x_0, b)$ contradicts \eqref{eq profile-shift3}. 
We thus have $\ul{W}_x(x) = W_x(x)$ for almost every $x \in (-b,b)$. 
Moreover, the monotonicity \eqref{mono-differential} holds with  $\phi(x)=\ul{W} - Ax^2/2$ and $b=b_0$ in the case that $\phi$ is concave, where $A$ and $b_0$ are the constants as above. 
These imply that $\ul{W}_x(x)$ exists and coincides with $W_x(x)$ for all $x \in (-b,b)$ since $W_x$ is continuous in $(-b,b)$. 

Our proof for the part on $\ol{W}$ is similar. In this case, instead of \eqref{eq profile-shift2} we have 
\[
G(\ol{W}_x(y))-G(\ol{W}_x(x))\geq f^{-1}(c) (y-x)
\]
at any differentiable points $x, y$ of $\ol{W}$ with $-b < x < y < b$, and instead of the singular boundary condition, we use the Lipschitz continuity in $(-b, b)$ to deduce that 
$\ol{W}_x=W_x$ in $(-b, b)$. We leave the details to the reader. 
\end{proof}

\begin{rmk}\label{rmk super conti}
We stress that under the assumptions of Theorem \ref{thm:profile-shift},  $\ul{W}=W+C$ in $[-b, b]$ for some $C\in \R$. Although by definition $\ul{W}$ is only lower semicontinuous in $[-b, b]$, $\ul{W}$ must be continuous up to the boundary for otherwise it will violate the singular boundary condition. 
Thus, the assumption $\ul{W} \in C([-b,b])$ in Theorem \ref{thm:profile-shift} can be removed. 
\end{rmk}

\begin{rmk}\label{rmk:Lpm}
Theorem \ref{thm:profile-shift} also implies that we can choose an initial value $u_0$ satisfying (A3) only if $L_- \le c \le L_+$.  
\end{rmk}

\section{Large time behavior}\label{sec:asymptotics}

In this section, we prove Theorem \ref{thm:conver-sol}.
Let us first establish several regularity results.

\begin{prop}[Local equi-Lipschitz regularity]\label{prop:local-lip}
Let $b>0$. 
Assume that functions $f$ and $g$ satisfy (A1) and (A2). 
Assume that $u_0 \in C([-b,b])$ is convex. 
Let $u$ be the unique solution of \eqref{flow eq}--\eqref{initial}. 
Then, for any $b_0 \in (0,b)$, there exists a constant $M>0$ such that 
\begin{equation}\label{local-lip} 
|u(x,t) - u(y,t)| \le M |x-y| \quad \mbox{for} \quad (x,t), (y,t) \in [-b_0, b_0] \times [0, \infty). 
\end{equation} 
\end{prop}

\begin{proof}
Let $(W,c)$ be the solution of \eqref{traveling eq} and \eqref{traveling bdry} obtained in Theorem \ref{thm:ex-tw}. 
Then, applying \eqref{initial dependence} together with the boundedness of $W$ and $u_0$, we can take  
\[
a_{\pm}=\pm \max_{[-b, b]} |u_0-W|
\]
such that 
\begin{equation}\label{asymptotic bound} 
W(x) + ct + a_- \le u(x,t) \le W(x) + ct + a_+ \quad \mbox{for} \quad (x,t) \in [-b, b] \times [0,\infty). 
\end{equation}
Thus, it follows again from the boundedness of $W$ that there exists a constant $\tilde{M} > 0$ such that 
\begin{equation}\label{bdd-uxuy}
|u(x,t) - u(y,t)| \le |W(x) - W(y) + a_+ - a_-| \le \tilde{M}
\end{equation}
for $x,y \in [-b,b]$ and $t \ge 0$. 
We define a constant $M>0$ as 
\begin{equation}\label{fix-M}
M := \dfrac{\tilde{M}}{(b-b_0)}. 
\end{equation}
Hence, we have by \eqref{bdd-uxuy} and \eqref{fix-M}
\begin{equation}\label{non-bdry-pt}
|u(\pm b,t) - u(y,t)| \le M |\pm b - y| \quad \mbox{for} \quad (y,t) \in [-b_0, b_0] \times [0, \infty). 
\end{equation}
Now, we take arbitrary points $x,y \in [-b_0, b_0]$ and assume $x < y$ without loss of generality. 
From Corollary \ref{cor:convex}, $u(\cdot,t)$ is convex for any $t\ge 0$ and hence a similar argument to obtain \eqref{mono-1} can be applied. 
Then, we have by \eqref{non-bdry-pt}
\[ - M \le \dfrac{u(x,t) - u(-b,t)}{x + b} \le \dfrac{u(y,t) - u(x,t)}{y-x} \le \dfrac{u(b,t) - u(y,t)}{b-y} \le M \]
for any $t \ge 0$, which is equivalent to \eqref{local-lip}. 
\end{proof}

An argument similar to Lemma \ref{lem:regularity-TW} enables us to obtain regularity in space higher than Lipschitz continuity provided that $u_0$ additionally satisfies (A3). We assume (A3) in order to obtain a time Lipschitz bound of the solution by Corollary \ref{cor:fixt}. 

\begin{prop}[Regularity improvement]\label{cor:regularity}
Let $b>0$. 
Assume that functions $f,g$ and $u_0$ satisfy (A1), (A2) and (A3) with $L_- \ge 0$. 
Let $u$ be the unique solution of \eqref{flow eq}--\eqref{initial}. 
Then, $u(\cdot,t) \in C^{1,1}((-b,b))$ for any $t>0$. 
\end{prop}
\begin{proof}
Proposition \ref{prop:local-lip} implies the existence of constants $A_1 \le A_2$ depending on $b_0 \in (0,b)$ such that $A_1 \le u_{xx}(x) \le A_2$ for $x \in (-b_0,b_0)$ in the viscosity sense by applying Corollary \ref{cor:fixt} and a similar argument to obtain \eqref{ine-Wxx}. 
Then, $u(\cdot,t)$ is semiconvex and semiconcave in $(-b_0,b_0)$. Hence Corollary \ref{cor:regular-convex} can be applied. 
\end{proof}

If (A3) holds with $L_- > 0$, then we can further prove the spatial H\"older continuity of the solution up to the boundary. 

\begin{prop}[Equi-H\"{o}lder regularity]
Let $b>0$. 
Assume that functions $f,g$ and $u_0$ satisfy (A1), (A2) and (A3) with $L_- > 0$. 
Let $u$ be the unique solution of \eqref{flow eq}--\eqref{initial}. 
Then, there exists a constant $M_\alpha>0$ such that 
\[ 
|u(x,t) - u(y,t)| \le M_\alpha |x-y|^{\frac{\alpha-2}{\alpha-1}} 
\]
for all $x, y\in [-b, b]$ and $t\geq 0$.
\end{prop}

\begin{proof}
Let $(W, c)$ be the solution of \eqref{traveling eq} and \eqref{traveling bdry} obtained in Theorem \ref{thm:ex-tw}.
Choose $c_\pm \in \mathbb{R}$ so that $0 < c_- < 1 < c_+$, $g((W_+)_x)(W_+)_{xx} = f^{-1}(L_-)$ and $g((W_-)_x)(W_-)_{xx} = f^{-1}(L_+)$, where 
\[ \begin{aligned}
&W_+(x) := c_+ W(-b + (x+b)/c_+) \quad \text{for} \quad x \in [-b, b], \\
&W_-(x) := c_- W(-b + (x+b)/c_-) \quad \text{for} \quad x \in [-b, 2c_- b - b]. 
\end{aligned} \]
Note that we can choose such sufficiently large $c_+$ and small $c_-$ since $g((W_\pm)_x)(W_\pm)_{xx} = f^{-1}(c)/c_\pm$. 
See also Remark \ref{rmk:Lpm} for the order $L_- \le c \le L_+$. 
From Corollary \ref{cor:fixt} and Proposition \ref{cor:regularity}, we have 
\[ \begin{aligned}
G(u_x(x,t)) - G(u_x(y,t)) =&\; \int_y^x (G \circ u_x)_x(z,t) \; dz \\
\ge&\; f^{-1}(L_-)(x-y) = G((W_+)_x(x)) - G((W_+)_x(y)) 
\end{aligned}\]
for $-b < y < x < b$ and $t \ge 0$. 
Letting $y \to -b$, we have by $u_x(y,t) \to -\infty$ and the monotonicity of $G$
\[ u_x(x,t) \ge (W_+)_x(x) \quad \text{for} \quad -b < x < b, \; \; t \ge 0. \]
Integrating the inequality in $[-b + \varepsilon, x]$ and letting $\varepsilon \to 0$, we obtain by the continuity of $u(\cdot,t)$ 
\begin{equation}\label{holder1}
u(-b,t) - u(x,t) \le W_+(-b) - W_+(x) \quad \text{for} \quad -b < x < b, \; \; t \ge 0. 
\end{equation}
Similary, we also obtain
\begin{equation}\label{holder2}
u(-b,t) - u(x,t) \ge W_-(-b) - W_-(x) \quad \text{for} \quad -b < x < 2c_- b - b, \; \; t \ge 0. 
\end{equation}
It follows from \eqref{holder1}, \eqref{holder2}, the H\"{o}lder continuity of $W$ and the choices of $c_\pm$ that there exists a constant $\tilde{M}_\alpha>0$ such that
\[ |u(-b,t) - u(x,t)| \le \tilde{M}_\alpha |x+b|^{\frac{\alpha-2}{\alpha-1}} \quad \mbox{for} \quad (x,t) \in [-b, 2c_- b - b] \times [0,\infty). \]
By a similar argument we can prove that 
\[ |u(b,t) - u(x,t)| \le \tilde{M}_\alpha |b-x|^{\frac{\alpha-2}{\alpha-1}} \quad \mbox{for} \quad (x,t) \in [b-2c_- b, b] \times [0,\infty). \]
The uniform H\"{o}lder continuity of $u(\cdot,t)$ follows from the H\"{o}lder continuity near the boundary points and the local Lipschitz continuity in $(-b,b)$. 
\end{proof}

Under the assumptions in Theorem \ref{thm:conver-sol}, due to \eqref{asymptotic bound}, for the function 
\[
w(x, t):=u(x, t)-ct, \quad (x, t)\in [-b, b]\times [0, \infty),
\]
we can define at every $x\in [-b, b]$ the following relaxed limits as $t\to \infty$:
\begin{equation}\label{profile limits}
\begin{aligned}
\ol{W}(x)&=\limsups_{t\to \infty} w(x, t)=\lim_{\vep\to 0+} \sup\{w(y, s): y\in [-b, b],\ |x-y|\leq \vep, \ s\geq 1/\vep\},\\
\ul{W}(x)&=\liminfs_{t\to \infty} w(x, t)=\lim_{\vep\to 0+} \inf\{w(y, s): y\in [-b, b],  |x-y|\leq \vep, \ s\geq 1/\vep\}.
\end{aligned}
\end{equation}

Adopting similar stability arguments in the proof of Theorem \ref{thm stability}, we show the following result. 
\begin{prop}[Sub- and supersolution properties of relaxed limits]\label{prop:profile-sol}
Let $b>0$. 
Assume that functions $f,g$ and $u_0$ satisfy (A1), (A2) and (A3) with $L_- > 0$. 
Let $u$ be the unique solution of \eqref{flow eq}--\eqref{initial}. 
Then, the relaxed limits $\ol{W}$ and $\ul{W}$ given as in \eqref{profile limits} are respectively a subsolution and a supersolution of \eqref{traveling eq} and \eqref{traveling bdry} with the speed $c$ of the traveling wave obtained in Theorem \ref{thm:ex-tw}. 
\end{prop}

\begin{proof}
Suppose that there exist $\phi\in C^2([-b, b])$ and $x_0\in (-b, b)$ such that $\ul{W}-\phi$ attains a strict minimum in $[-b, b]$ at $x_0$. Then by definition there exists $(y_\vep, s_\vep)\in (-b, b)\times [1/\vep, \infty)$ such that $y_\vep\to x_0$ as $\vep\to 0$ and
\[
w(y_\vep, s_\vep)-\phi(y_\vep)\leq \ul{W}(x_0)-\phi(x_0)+\omega(\vep),
\]
where $\omega$ is a modulus of continuity. Letting
\[
\Phi(x, t):=w(x, t)-\phi(x)+\vep(t-s_\vep)^2,
\]
we have
\[
\inf_{(x, t)\in Q}\Phi(x, t)\leq \Phi(y_\vep, s_\vep)\leq \ul{W}(x_0)-\phi(x_0)+\omega(\vep),
\]
where we recall that $Q=(-b, b)\times (0, \infty)$. It follows that $\Phi$ attains a minimum in $\overline{Q}$ at some $(x_\vep, t_\vep)\in (-b, b)\times (0, \infty)$ satisfying $x_\vep\to x_0$ as $\vep\to 0$. Indeed, since $w(x, t)-\phi(x)$ is bounded in $\ol{Q}=[-b, b]\times [0, \infty)$,  we get
\[
\inf_{x\in [-b, b]}\Phi(x, 0)\geq \inf_{x\in [-b, b]} (w(x, 0)-\phi(x))+1/\vep> \ul{W}(x_0)-\phi(x_0)+\omega(\vep)
\]
when $\vep>0$ is taken sufficiently small and also have
\[
\inf_{x\in [-b, b]}\Phi(x, t)\to \infty\quad \text{as $t\to \infty$.}
\]
We thus can find $(x_\vep, t_\vep)\in [-b, b]\times (0, \infty)$ such that 
\[
\min_{\overline{Q}} \Phi=\Phi(x_\vep, t_\vep),
\]
which implies that $t_\vep\to \infty$ as $\vep\to 0$ and
\begin{equation}\label{eq profile-sol1}
\Phi(x_\vep, t_\vep)\leq \Phi(y_\vep, s_\vep)\leq \ul{W}(x_0)-\phi(x_0)+\omega(\vep).
\end{equation}
By taking a subsequence, we have $x_\vep\to \hat{x}$ as $\vep\to 0$. It follows that 
\[
\ul{W}(\hat{x})-\phi(\hat{x})\leq \liminf_{\vep\to 0} \Phi(x_\vep, t_\vep)\leq \ul{W}(x_0)-\phi(x_0).
\]
Noticing that $\ul{W}-\phi$ attains a strict a minimum at $x_0$, we have $\hat{x}=x_0$, which means that $x_\vep\to x_0$ as $\vep\to 0$. In particular, $x_\vep\neq \pm b$ when $\vep>0$ is small. 

Moreover, \eqref{eq profile-sol1} yields
\[
\limsup_{\vep\to 0}\vep(t_\vep-s_\vep)^2\leq \ul{W}(x_0)-\liminf_{\vep\to 0}w(x_\vep, t_\vep)\leq 0,
\] 
which further implies that 
\begin{equation}\label{eq profile-sol2}
\vep|t_\vep-s_\vep|\to 0 \quad \text{as} \; \; \vep\to 0.
\end{equation}
Since 
\[
(x, t)\mapsto u(x, t)-ct+\vep(t-s_\vep)^2-\phi(x)
\]
has a minimum at $(x_\vep, t_\vep)$, applying Definition \ref{def singular}, we obtain
\[
c-2\vep(t_\vep-s_\vep)+F(\phi_x(x_\vep),\ \phi_{xx}(x_\vep))\geq 0.
\]
Letting $\vep\to 0$ and adopting \eqref{eq profile-sol2}, we get
\[
c+F(\phi_x(x_0),\ \phi_{xx}(x_0))\geq 0
\]
as desired. 

We next verify the singular boundary condition. We show that $\ul{W}$ cannot be tested from below by $C^2$ functions at $x=\pm b$. Assume by contradiction that there exists $\phi\in C^2([-b, b])$ such that $\ul{W}(x)-\phi(x)$ attains a minimum at $x=b$.  (The case when the minimum attains at $x=-b$ can be treated in a symmetric way.) We slightly change the test function by setting 
\[
\phi_\delta(x)=\phi(x)+\delta(x-b)+{1\over \delta}(x-b)^2 
\]
with $\delta>0$ for all $x\in [-b, b]$. It is not difficult to see that for each $\delta>0$, $\ul{W}-\phi_\delta$  attains a strict local minimum at $b$. 

We can follow the same argument as in the interior case above (with $x_0$ replaced by $b$) to find, for each $\vep>0$,  a local minimizer $(x_\vep, t_\vep)$ of 
\[
(x, t)\mapsto u(x, t)-ct-\phi_\delta(x)+\vep(t-s_\vep)^2
\]
for some $s_\vep\geq 1/\vep$. Also, we still have $x_\vep\to b$ and \eqref{eq profile-sol2} as $\vep\to 0$. However, by Definition \ref{def singular} we get $x_\vep\neq b$, since $w$ satisfies the singular boundary condition. Applying the definition of supersolutions again and letting $\vep\to 0$, we deduce that
\[
c+F\left(\phi_x(b)+\delta,\ \phi_{xx}(b)+{2\over \delta}\right)\geq 0.
\]
Sending $\delta\to 0$, we obtain a contradiction due to the fact that
\[
F\left(\phi_x(b)+\delta,\ \phi_{xx}(b)+{2\over \delta}\right)\to -\infty\quad \text{as $\delta\to 0$.}
\]

We have completed the verification of the supersolution $\ul{W}$. We omit the analogous proof for $\ol{W}$. In fact, there is no need to consider the boundary condition for subsolutions and we only need to apply the interior argument above in a symmetric manner. 
\end{proof}

For our later application, below we prepare a strong comparison principle with smooth sub- or supersolutions by adapting the arguments in \cite{Dal}.  

\begin{prop}[Strong comparison principle]\label{prop strong cp}
Let $b>0$. 
Assume that functions $f, g$ satisfy (A1) and (A2). 
Assume also $f^{-1}$ is Lipschitz away from $s=0$ and $g$ is Lipschitz in $\mathbb{R}$. 
Let $u$ and $v$ be respectively a sub- and a supersolution of \eqref{flow eq}. 
Assume that either $u$ or $v$ is of class $C^2(Q)$ with its time derivative uniformly bounded away from $0$. 
If there exists $(x_0, t_0)\in Q$ such that $u-v$ attains a maximum at $(x_0, t_0)$, then $u(\cdot, t_0)-v(\cdot, t_0)$ is constant in $(-b, b)$. 
\end{prop}

\begin{proof}
Let $z_0=(x_0, t_0)$. Suppose that $u-v$ attains a maximum at $z_0$. Let us assume that $v\in C^2(Q)$ and 
\beq\label{strong cp3}
|v_t|\geq \mu \quad \text{in $Q$}
\eeq
for some $\mu>0$. The case when $u$ satisfies the same conditions can be similarly handled. Moreover, without loss of generality we may assume that there exists $\delta>0$ small such that $u(x, t_0)-v(x, t_0)< u(z_0)-v(z_0)$ for all $x\in (x_0, x_0+2\delta)$ with $x_0+2\delta<b$. 
Let $x_\delta=x_0+\delta$ and $z_\delta=(x_\delta, t_0)$.
 Then there exists $r\in (0, \delta]$ and an ellipse 
\[
\O_r:=\{(x, t)\in Q:  (x-x_\delta)^2+a(t-t_0)^2< r^2\} 
\]
with $a>1$ such that 
\[
u-v< u(z_0)-v(z_0)\quad \text{in $\O_r(z_\delta)$}
\]
and 
\beq\label{strong cp7}
u(\hat{z})-v(\hat{z})=u(z_0)-v(z_0),
\eeq
for some $\hat{z}=(\hat{x}, \hat{t})\in \partial \O_r$. Moreover, by choosing $a>1$ sufficiently large, we can let all such $\hat{z}$ satisfy 
\beq\label{strong cp5}
|\hat{x}-x_\delta|\geq {2r\over 3}.
\eeq
Fix a large constant $\gamma>0$ to be determined later and arbitrarily take $0<\vep\leq 1/\gamma^2$. We consider 
\[
\Psi_\vep(x, t)=u(x, t)-v(x, t)+\psi^\vep(x, t)
\]
for $(x, t)\in Q$, where $\psi^\vep$ is given by 
\[
\psi^\vep(x, t)=\vep e^{-\gamma(x-x_\delta)^2-a\gamma(t-t_0)^2}.
\]
It is not difficult to see that, for any $\vep>0$ small, 
$\Psi_\vep$ attains a local maximum at some $z_\vep=(x_\vep, t_\vep)\in Q$ near 
some $\hat{z}\in \partial \O_r$ fulfilling \eqref{strong cp7}. Let us take a subsequence, still indexed by $\vep$ for simplicity, and assume that $z_\vep\to \hat{z}$ as $\vep\to 0$. 
As a consequence, for any $\vep>0$ small we have 
\beq\label{strong cp6}
\max\{|x_\vep-x_\delta|,\  |t_\vep-t_0|\}\leq \delta
\eeq
and by \eqref{strong cp5}, 
\beq\label{strong cp4}
|x_\vep-x_\delta|\geq {r\over 2}.
\eeq
Since $u$ and $v$ are respectively a sub- and a supersolution of \eqref{flow eq}, we have 
\beq\label{strong cp1}
f^{-1}(v_t(z_\vep)-\psi^\vep_t(z_\vep))\leq g\left(v_x(z_\vep)-\psi^\vep_x(z_\vep)\right)(v_{xx}(z_\vep)-\psi^\vep_{xx}(z_\vep)),
\eeq
\beq\label{strong cp2}
f^{-1}(v_t(z_\vep))\geq g\left(v_x(z_\vep)\right)v_{xx}(z_\vep).
\eeq
By direct computation, we have
\[
\begin{aligned}
\psi^\vep_t(z_\vep)&=-2 a\gamma (t_\vep-t_0) \psi^\vep(z_\vep),\\
\psi^\vep_x(z_\vep)&=-2\gamma (x_\vep-x_0)\psi^\vep(z_\vep),\\
\psi^\vep_{xx}(z_\vep)& =-2\gamma \psi^\vep(z_\vep)+4\gamma^2 (x_\vep-x_0)^2\psi^\vep(z_\vep). 
\end{aligned}
\]
Note that 
\[
|\psi_{xx}^\vep(z_\vep)|\leq C_0
\]
holds for some $C_0>0$ depending on $\delta>0$, thanks to \eqref{strong cp6} and the relation that $\vep<1/\gamma^2$.
Similary, we can also see that
\[ |\psi_t^\vep(z_\vep)| \le \mu/2 \]
for sufficiently large $\gamma$, thanks to \eqref{strong cp6} and the relation that $\vep<1/\gamma^2$. 
In view of \eqref{strong cp3}, by using the Lipschitz regularity of $f^{-1}$ away from $s=0$, we have
\[
f^{-1}(v_t(z_\vep)-\psi^\vep_t(z_\vep))-f^{-1}(v_t(z_\vep)) \geq 
-C_1 \gamma \psi^\vep(z_\vep),
\]
where $C_1>0$ depends on $a$, $\alpha$, $\delta$ and $\mu$.  
Also, the Lipschitz continuity of $g$ implies the existence of $C_2>0$ depending on $\alpha$, $b$ and $|v_x(z_\vep)|$ such that
\[
\left|g\left(v_x(z_\vep)-\psi^\vep_x(z_\vep)\right)-g\left(v_x(z_\vep)\right) \right|
\leq C_2\gamma \psi^\vep(z_\vep).
\]
Taking the difference between \eqref{strong cp1} and \eqref{strong cp2}, we thus get
\[
\begin{aligned}
-C_1\gamma \psi^\vep (z_\vep)&\leq \left|g\left(v_x(z_\vep)-\psi^\vep_x(z_\vep)\right)-g\left(v_x(z_\vep)\right) \right| \left|v_{xx}(z_\vep)-\psi^\vep_{xx}(z_\vep)\right|
-g(v_x(z_\vep))\psi_{xx}^\vep(z_\vep)\\
&\leq C_2\gamma \psi^\vep(z_\vep)( |v_{xx}(z_\vep)|+C_0)-g\left(v_x(z_\vep)\right)
 \psi^\vep_{xx}(z_\vep)\\
&\leq C\gamma \psi^\vep(z_\vep)-{1\over C}\gamma^2 (x_\vep-x_0)^2\psi^\vep(z_\vep),
\end{aligned}
\]
where $C>0$ is a large constant depending only on $\alpha$, $b$, $C_0$ as well as the local bound of $|v_x|$ and $|v_{xx}|$ near $\hat{z}$. 
Due to \eqref{strong cp4}, the estimate above yields
\[
-C_1  \leq C-{r^2\over 4C}\gamma, 
\]
which is a contradiction if we take 
\[
\gamma>{4C(C_1+C)\over r^2}.
\]
Our proof is thus complete. 
\end{proof}

The lower bounded for the evolution speed of $u$ or $v$ seems to be a necessary condition, since the parabolic degeneracy of \eqref{flow eq} occurs exactly at $u_t=0$. 

We finally proceed to the proof of Theorem \ref{thm:conver-sol}.

\begin{proof}[Proof of Theorem \ref{thm:conver-sol}]
We first prove the convergence result for $u_0$ satisfying (A3) with $L_->0$. We later handle the case with general initial data by approximation. 

For any $t\geq 0$, let 
\[
\begin{aligned}
 \ol{\eta}(t) &= \max_{x \in [-b,b]} \left\{u(x, t) - W(x) - ct\right\},\\
\ul{\eta}(t) &= \min_{x \in [-b,b]} \left\{u(x, t) - W(x) - ct\right\}.
\end{aligned}
\]
Applying the comparison principle in Theorem \ref{thm cp}, we see that $\ol{\eta}$ is non-increasing while $\ul{\eta}$ is non-decreasing in $[0, \infty)$. As a result, we have 
\[
\ol{\eta}(t)\to \ol{m}, \quad \ul{\eta}(t)\to \ul{m}
\]
as $t\to \infty$. In order to prove the desired convergence result, we only need to show that $\ol{m}=\ul{m}$. \\
\mbox{}\\
If we view $u(x, t+s)-W(x)-ct-cs$ as a family of functions of $(x, t)$ in $\ol{Q}$ parametrized by $s\geq 0$, then it is not difficult to see that it is uniformly bounded and H\"{o}lder continuous in $\ol{Q}$. By the Arzel\`a-Ascoli theorem, we can find a sequence $s_k>0$ and $w\in C(\ol{Q})$ such that 
\[
u(x, t+s_k)-c(t+s_k)\to w(x, t)
\]
uniformly in $\ol{Q}_T$ as $k\to \infty$ for any $T>0$.  In addition, $w(\cdot, t)$ is H\"{o}lder continuous in $[-b, b]$ and locally Lipschitz in $(-b, b)$ uniformly for all $t>0$.
By Theorem \ref{thm stability}, we can show that $v(x, t)=w(x, t)+ct$ is a solution of \eqref{flow eq} and \eqref{bdry}. Moreover, sending $k\to \infty$ for $\ol{\eta}(t+s_k)$ and $\ul{\eta}(t+s_k)$ leads to
\[
\begin{aligned}
\ol{m}&= \max_{x \in [-b,b]} \left\{w(x, t)- W(x) \right\},\\
\ul{m}&=\min_{x \in [-b,b]} \left\{w(x, t)- W(x) \right\},
\end{aligned}
\] 
for all $t> 0$. 

Suppose that there exists $(x_0, t_0)\in Q$ such that
\[
\ol{m}= w(x_0, t_0)- W(x_0)=v(x_0, t_0)-ct_0-W(x_0). 
\]
Since $W(x)+ct$ is a smooth solution of \eqref{flow eq} and \eqref{bdry} with its time derivative being $c>0$, by Proposition \ref{prop strong cp} we deduce that $w(\cdot, t_0)-W$ is a constant in $[-b, b]$ and therefore $\ol{m}=\ul{m}$. 
We can use the same argument to prove $\ol{m}=\ul{m}$ if the minimum for $\ul{m}$ is attained 
in $Q$.

We finally consider the case when $\ol{m}$ and $\ul{m}$ are only attained at $\{\pm b\}$ for any $t>0$. 
Assume by contradiction that $\ol{m}>\ul{m}$. By continuity of $w$, without loss of generality we may assume that
\beq\label{convergence bdry1}
\begin{aligned}
\ol{m}& = w(b, t)-W(b),\\
\ul{m}&= w(-b, t)-W(-b) 
\end{aligned}
\eeq
for all $t>0$. In view of Proposition \ref{prop:profile-sol}, we see that $\ul{W}$ given in  \eqref{profile limits} is a supersolution of \eqref{traveling eq} and \eqref{traveling bdry}. 
Hence, by \eqref{convergence bdry1} and the H\"{o}lder continuity of $w(\cdot, t)$ uniformly for all $t>0$, it follows that 
\[
\begin{aligned}
\ol{m}&=\ul{W}(b)-W(b), \\
\ul{m}&=\ul{W}(-b)-W(-b),
\end{aligned}
\]
which yields 
\beq\label{convergence bdry2}
\ul{W}(b)-W(b)>\ul{W}(-b)-W(-b)
\eeq
Note that the local Lipschitz continuity of $w(\cdot, t)$ for all $t>0$ implies the local Lipschitz continuity of $\ul{W}$. Therefore, in light of Theorem \ref{thm:profile-shift} together with Remark \ref{rmk super conti},  $\ul{W}-W$ is constant in $[-b, b]$, which is clearly a contradiction to \eqref{convergence bdry2}.

It remains to consider the general case when $u_0\in C([-b, b])$ is convex.  In this case, we apply the approximation of $u_0$ given in the proof of Theorem \ref{thm existence perron}. 
Let $\psi_{\varepsilon} \in C([-b,b]) \cap C^\infty((-b,b))$ be the approximation of $u_0$ defined as in Lemma \ref{lem regularization}.
Applying Theorem \ref{thm:conver-sol}, we see that the unique solution $u_\vep\in C(\ol{Q})$ corresponding to the initial value $\psi_\vep$ satisfies 
\[
u_\vep(x, t)-ct-W(x)\to m_\vep\quad \text{as $t\to \infty$}
\]
uniformly for all $x\in [-b, b]$, where $m_\vep\in \R$ is a constant depending on $\psi_\vep$. 
For any subsequence along $\vep_j$, we write $u_j=u_{\vep_j}$,  $\psi_j=\psi_{\vep_j}$ and $m_j=m_{\vep_j}$ for $j\geq 1$. We can utilize \eqref{initial dependence} to get 
\[
\sup_{\ol{Q}}|u-u_j|\leq \max_{[-b, b]}|u_0-\psi_j| \to 0\quad \text{as $j\to \infty$},
\]
which implies that 
\[
\sup_{\ol{Q}}|u_j-u_k|\to 0\quad \text{as $j, k\to \infty$.} 
\]
This yields $|m_j-m_k|\to 0$ as $j, k\to \infty$. In other words, $m_j$ is a Cauchy sequence.
Thus we get $m\in \R$ such that $m_j\to m$ as $j\to \infty$. Noticing that
\[
|u(x, t)-ct-W(x)-m|\leq \sup_{\ol{Q}}|u-u_j|+|m_j-m|+|u_j(x, t)-ct-W(x)-m_j|
\]
holds for all $(x, t)\in \ol{Q}$, we end up with \eqref{convergence intro} for general initial data $u_0$ by letting $t\to \infty$ and then $j\to \infty$. 
\end{proof}

\appendix 
\section{Several Properties of Convex Functions}

The results in this section are elementary and well known but we present them here to improve the readability. 
In the proofs, we use the one-sided derivatives defined by 
\begin{equation}\label{def-one-side}
\begin{aligned}
&\phi_x^+(x) := \lim_{y \to x+} \dfrac{\phi(y) - \phi(x)}{y-x}, \\
&\phi_x^-(x) := \lim_{y \to x-} \dfrac{\phi(y) - \phi(x)}{y-x} 
\end{aligned}
\end{equation}
for $\phi \in C([-b,b])$ and $x \in (-b,b)$ if they exist. 

\begin{lem}\label{lem:mono-differential}
Let $b >0$. 
Assume that $\phi \in C([-b,b])$ is convex (resp.\ concave). 
Then, one-sided derivatives defined by \eqref{def-one-side} are well-defined in $(-b,b)$, finite and 
\begin{equation}\label{mono-differential} 
\phi^-_x(x) \le \phi^+_x(x) \le \phi^-_x (y) \le \phi^+_x(y) \quad (\text{resp.\ } \phi^-_x(x) \ge \phi^+_x(x) \ge \phi^-_x (y) \ge \phi^+_x(y)) 
\end{equation}
for $-b < x < y < b$. 
\end{lem}

\begin{proof}
Assume that $\phi$ is convex. 
The proof on a concave function $\phi$ is similar. 
Let $s, t, u, v \in [-b,b]$ be arbitrary such that $-b \le s < u < v < t \le b$. 
Then, we obtain 
\begin{equation}\label{mono-1}
 \dfrac{\phi(u) - \phi(s)}{u-s} \le \dfrac{\phi(t)-\phi(u)}{t-u} \le \dfrac{\phi(t) - \phi(v)}{t-v}. 
\end{equation}
Indeed, it follows from the convexity of $\phi$ that 
\[ \phi(u) \le \dfrac{u-s}{t-s} \phi(t) + \dfrac{t-u}{t-s} \phi(s) \]
which is equivalent to the first inequality in \eqref{mono-1}. 
The second inequality can be proved by a similar argument. 
Taking $t=x \in (-b,b)$ in the second inequality, we can see that $(\phi(y) - \phi(x))/(y-x)$ is non-decreasing as $y \to x-$ and hence the one-sided derivative $\phi^-_x(x)$ exists. 
Similarly, $\phi^+_x(x)$ exists for $x \in (-b,b)$. 
The first inequality in \eqref{mono-1} with $u=x \in (-b,b)$ and the inequality \eqref{mono-1} with $-b < s=x < y=t < b$ yield $\phi^-_x(x) \le \phi^+_x(x)$ and $\phi^+_x(x) \le \phi^-_x(y)$. 
Thus, we have \eqref{mono-differential}. 
Similarly, we can obtain 
\begin{equation}\label{bdd-derivatives} 
\dfrac{\phi(x) - \phi(-b)}{x+b} \le \phi^-_x(x) \le \phi^+_x(x) \le \dfrac{\phi(b) - \phi(x)}{b-x} 
\end{equation}
for $x \in (-b,b)$, thus the one-sided derivatives are finite in $(-b,b)$. 
\end{proof}

\begin{cor}\label{cor:regular-convex}
Let $b>0$ and $\phi \in C([-b,b])$.
Assume that for any $b_0 \in (0,b)$ there exists $A_1 \ge A_2$ such that $\phi(x) + A_1 x^2/2$ is convex and $\phi(x) + A_2 x^2/2$ is concave in $[-b_0,b_0]$. 
Then, $\phi \in C^{1,1}((-b,b))$ and 
\begin{equation}\label{lip-deri} 
A_2 (y-x) \le \phi_x(x) - \phi_x(y) \le A_1(y-x) \quad \text{for} \quad -b_0 < x \le y < b_0. 
\end{equation}
\end{cor}

\begin{proof}
Fix an arbitrary $b_0 \in (0,b)$. 
Applying Lemma \ref{lem:mono-differential} to $\phi(x) + A_1 x^2/2$ and $\phi(x) + A_2 x^2/2$, we obtain $\phi^-_x(x) \le \phi^+_x(x)$ and $\phi^-_x(x) \ge \phi^+_x(x)$ at any point $x \in (-b_0,b_0)$, respectively. 
Thus, $\phi \in C^1((-b_0,b_0))$. 
Lemma \ref{lem:mono-differential} also yields \eqref{lip-deri}. 
Since $b_0$ is arbitrary, $\phi \in C^{1,1}((-b,b))$. 
\end{proof}

\begin{lem}\label{lem:mono-integral}
Let $b>0$. 
Assume that $G \in C^{0,1}(\mathbb{R})$ is monotone increase.
Let $\phi \in C([-b,b])$. 
Assume that for $b_0 \in (0,b)$ there exists $A \in \mathbb{R}$ such that $\phi(x) + Ax^2/2$ is convex (resp.\ concave) in $[-b_0,b_0]$.   
Then, $\phi_x$ and $(G(\phi_x(x)))_x$ exists almost everywhere in $(-b,b)$ and 
\begin{equation}\label{mono-integral}
\begin{aligned}
&G(\phi_x(x)) - G(\phi_x(y)) \ge \int_y^x (G\circ\phi_x)_x(z) \; dz \\ 
&\left(\text{resp.\ } G(\phi_x(x)) - G(\phi_x(y)) \le \int_y^x (G\circ\phi_x)_x(z)\right)
\end{aligned}
\end{equation}
for $-b < y < x < b$ if $\phi_x$ exists at $x$ and $y$. 
\end{lem}

\begin{proof}
Fix an arbitrary $b_0 \in (0,b)$. 
Assume that $\phi(x)+Ax^2$ is convex for some $A \in \mathbb{R}$. 
The proof is similar in the case that $\phi(x) + Ax^2$ is concave. 
Let $b_1 \in (0,b_0)$. 
It is follows from \eqref{bdd-derivatives} that $\phi^-_x(x) + Ax$ and $\phi^+_x(x) + Ax$ is bounded in $[-b_1, b_1]$. 
Applying the monotonicity \eqref{mono-differential}, we can see that the derivative $\phi_x(x) + Ax$ is well-defined almost everywhere and non-decreasing in $x \in [-b_1, b_1]$. 
Thus, $\phi_x(x) + Ax$ is of bounded variation and hence there exist an absolutely continuous function $\phi_{\rm reg}$ and a step function $\phi_{\rm sing}$ with countable jumping points such that 
\[ \phi_x(x) = \phi_{\rm reg}(x) + \phi_{\rm sing}(x) \]
if $\phi_x$ exists at $x \in [-b_1, b_1]$. 
Note that $\phi_{\rm sing}$ is non-decreasing since 
\[ \begin{aligned}
\lim_{y_n \to x+} \phi_{\rm sing}(y_n) - \lim_{z_n \to x-} \phi_{\rm sing}(z_n) =&\; \lim_{y_n \to x+} (\phi_x(y_n) - \phi_{\rm reg}(y_n)) - \lim_{z_n \to x-} (\phi_x(z_n) - \phi_{\rm reg}(z_n)) \\
=&\;  \lim_{y_n \to x+} (\phi_x(y_n) - Ay_n) - \lim_{z_n \to x-} (\phi_x(z_n) - Az_n) \ge 0 
\end{aligned}\]
for any $x \in (-b_1,b_1)$ and sequences $\{y_n\}, \{z_n\}$ in the set of differentiable points of $\phi$. 

Since $G$ is Lipschitz and $\phi_x$ is of bounded variation, $G(\phi_x)$ is also of bounded variation and hence there exist an absolutely continuous function $G_{\rm reg}$ and a step function $G_{\rm sing}$ such that 
\[ G(\phi_x(x)) = G_{\rm reg}(x) + G_{\rm sing}(x) \]
if $\phi_x$ exists at $x \in [-b_1, b_1]$. 
It follows from the monotonicity of $G$ and $\phi_{\rm sing}$ that 
\[ \begin{aligned}
& \lim_{y_n \to x+} G_{\rm sing}(y_n) - \lim_{z_n \to x-} G_{\rm sing}(z_n) = \lim_{y_n \to x+} G(\phi_x(y_n)) - \lim_{z_n \to x-} G(\phi_x(z_n)) \\
=&\; \lim_{y_n \to x+} G(\phi_{\rm reg}(x) + \phi_{\rm sing}(y_n)) - \lim_{z_n \to x-} G(\phi_{\rm reg}(x) + \phi_{\rm sing}(z_n)) \ge 0 
\end{aligned} \]
thus $G_{\rm sing}$ is non-decreasing. 
Since $G(\phi_x)$ and $G_{\rm reg}$ has same derivative almost everywhere in $(-b_1, b_1)$, the fundamental theorem of calculus for $G_{\rm reg}$ and the monotonicity of $G_{\rm sing}$ implies 
\[ G(\phi_x(x)) - G(\phi_x(y)) \ge G_{\rm reg}(x) - G_{\rm reg}(y) = \int_y^x (G_{\rm reg})_x(z) \; dz = \int_y^x (G\circ\phi_x)_x(z) \; dz. \]
for $-b_1 < y < x < b_1$ if $\phi_x$ exists at $x$ and $y$. 
Since $b_1$ and $b_0$ are arbitrary so that $0 < b_1 < b_0 < b$, we have \eqref{mono-integral}. 
\end{proof}



\end{document}